\documentclass[11pt]{article}

\usepackage{amsmath, amsfonts, amssymb, amsthm,  graphicx, mathtools, enumerate,amsbsy, dsfont}

\usepackage[blocks, affil-it]{authblk}

\usepackage{mathrsfs}

\usepackage[numbers, square, sort]{natbib}
\usepackage[CJKbookmarks=true,
            bookmarksnumbered=true,
			bookmarksopen=true,
			colorlinks=true,
			citecolor=red,
			linkcolor=blue,
			anchorcolor=red,
			urlcolor=blue]{hyperref}
\usepackage[usenames]{color}

\usepackage[letterpaper, left=1.2truein, right=1.2truein, top = 1.2truein, bottom = 1.2truein]{geometry}
\usepackage[ruled, vlined, lined, commentsnumbered]{algorithm2e}

\usepackage{prettyref,soul}

\newtheorem{lemma}{Lemma}[section]

\newtheorem{thm}{Theorem}[section]

\newrefformat{eq}{(\ref{#1})}
\newrefformat{chap}{Chapter~\ref{#1}}
\newrefformat{sec}{Section~\ref{#1}}
\newrefformat{algo}{Algorithm~\ref{#1}}
\newrefformat{fig}{Fig.~\ref{#1}}
\newrefformat{tab}{Table~\ref{#1}}
\newrefformat{rmk}{Remark~\ref{#1}}
\newrefformat{clm}{Claim~\ref{#1}}
\newrefformat{def}{Definition~\ref{#1}}
\newrefformat{cor}{Corollary~\ref{#1}}
\newrefformat{lmm}{Lemma~\ref{#1}}
\newrefformat{lemma}{Lemma~\ref{#1}}
\newrefformat{prop}{Proposition~\ref{#1}}
\newrefformat{app}{Appendix~\ref{#1}}
\newrefformat{ex}{Example~\ref{#1}}
\newrefformat{exer}{Exercise~\ref{#1}}
\newrefformat{soln}{Solution~\ref{#1}}
\newrefformat{cond}{Condition~\ref{#1}}

\def\text#1{\mbox{\rm #1}}

\def\T{\mathcal{T}}

\def\T{{ \mathrm{\scriptscriptstyle T} }}
\def\H{{ \mathrm{\scriptscriptstyle H} }}
\def\n{{ \mathcal{N} }}
\def\cn{{ \mathcal{CN} }}

\newcommand{\argmin}{\mathop{\rm argmin}}

\newcommand{\norm}[1]{\left\|{#1} \right\|}

\newcommand{\wh}{\widehat}
\newcommand{\wt}{\widetilde}

\newcommand{\fnorm}[1]{\|#1\|_{\rm F}}
\newcommand{\opnorm}[1]{\|#1\|_{\rm op}}

\newcommand{\Tr}{\mathop{\sf Tr}}
\newcommand{\diag}{\mathop{\text{diag}}}

\newtheorem*{condb'}{Condition B'}

\newcommand{\br}[1]{\left( #1 \right)}

\newcommand{\abs}[1]{\left| #1 \right|}

\newcommand{\im}{{\rm Im}}
\newcommand{\re}{{\rm Re}}

\newcommand{\diff}{{\rm d}}

\title{SDP Achieves Exact Minimax Optimality in Phase Synchronization
}
\author[1]{Chao Gao}
\author[2]{Anderson Y. Zhang}
\affil[1]{
University of Chicago
}
\affil[2]{
University of Pennsylvania
}

\begin{document}
\maketitle

\begin{abstract}
We study the phase synchronization problem with noisy measurements $Y=z^*z^{*\H}+\sigma W\in\mathbb{C}^{n\times n}$, where $z^*$ is an $n$-dimensional complex unit-modulus vector and $W$ is a complex-valued Gaussian random matrix. It is assumed that each entry $Y_{jk}$ is observed with probability $p$. We prove that an SDP relaxation of the MLE achieves the error bound $(1+o(1))\frac{\sigma^2}{2np}$ under a normalized  squared $\ell_2$ loss. This result matches the minimax lower bound of the problem, and even the leading constant is sharp. The analysis of the SDP is based on an equivalent non-convex programming whose solution can be characterized as a fixed point of the generalized power iteration lifted to a higher dimensional space. This viewpoint unifies the proofs of the statistical optimality of three different methods: MLE, SDP, and generalized power method. The technique is also applied to the analysis of the SDP for $\mathbb{Z}_2$ synchronization, and we achieve the minimax optimal error $\exp\left(-(1-o(1))\frac{np}{2\sigma^2}\right)$ with a sharp constant in the exponent.

\smallskip

\end{abstract}


\section{Introduction}

Consider the problem of phase synchronization \citep{singer2011angular} with observations
\begin{equation}
Y_{jk}=z_j^*\bar{z}_k^*+\sigma W_{jk}\in\mathbb{C},\label{eq:basic-model}
\end{equation}
for $1\leq j<k\leq n$, where $\bar{z}_k^*$ stands for the complex conjugate of $z_k^*$. Our goal is to estimate $z_1^*,\cdots,z_n^*\in\mathbb{C}_1=\{x\in\mathbb{C}:|x|=1\}$. Since $|z_j^*|=1$, we can write $z_j^*=e^{i\theta_j^*}$ with some $\theta_j^*\in(0,2\pi]$ for all $j\in[n]$, and thus $Y_{jk}$ is understood to be a noisy observation of the pairwise difference between two angles $\theta_j^*$ and $\theta_k^*$. Following \cite{abbe2017group, bandeira2017tightness, gao2020multi, zhong2018near}, we consider an additive noise model and we assume that $W_{jk}$ is a standard complex Gaussian variable independently for all $1\leq j<k\leq n$.\footnote{For $W_{jk}\sim \cn(0,1)$, we have $\re(W_{jk})\sim \n\left(0,\frac{1}{2}\right)$ and $\im(W_{jk})\sim \n\left(0,\frac{1}{2}\right)$ independently.} 

Recently, minimax risk of estimating $z^*\in\mathbb{C}_1^n$ has been studied by \cite{gao2020exact} under the loss function
\begin{equation}
\ell(\wh{z},z)=\min_{a\in\mathbb{C}_1}\frac{1}{n}\sum_{j=1}^n|\wh{z}_j-z_ja|^2. \label{eq:loss-basic-z}
\end{equation}
We note that the minimization of $a\in\mathbb{C}_1^n$ in the definition of (\ref{eq:loss-basic-z}) is necessary, since a global rotation of the angles $\theta_1^*,\cdots,\theta_n^*$ does not change the distribution of the observations $\{Y_{jk}\}_{1\leq j<k\leq n}$. It was proved by \cite{gao2020exact} that the minimax risk of phase synchronization has the following lower bound
\begin{equation}
\inf_{\wh{z}\in\mathbb{C}_1^n}\sup_{z\in\mathbb{C}_1^n}\mathbb{E}_z\ell(\wh{z},z)\geq (1-o(1))\frac{\sigma^2}{2n}, \label{eq:lower-intro}
\end{equation}
and the maximum likelihood estimator (MLE), defined as a global maximizer of
\begin{equation}
\max_{z\in\mathbb{C}_1^n}z^{\H}Yz, \label{eq:basic-MLE}
\end{equation}
is proved to achieve the error bound $(1+o(1))\frac{\sigma^2}{2n}$, and is therefore asymptotically minimax optimal. However, the optimization problem (\ref{eq:basic-MLE}) is a  constrained quadratic programming that is generally known to be NP-hard. This motivates researchers to consider a convex relaxation of (\ref{eq:basic-MLE}) in the form of semi-definite programming (SDP) \citep{singer2011angular,bandeira2014open,bandeira2017tightness,zhong2018near,ling2020solving}. Write $Z=zz^{\H}$. For any $z\in\mathbb{C}_1^n$, $Z$ is a complex positive-semidefinite Hermitian matrix whose diagonal entries are all one. The SDP relaxation of (\ref{eq:basic-MLE}) is then defined as
\begin{equation}
\max_{Z=Z^{\H}\in\mathbb{C}^{n\times n}}\Tr(YZ)\quad\text{subject to}\quad \diag(Z)=I_n\text{ and }Z\succeq 0. \label{eq:SDP-intro}
\end{equation}
A global maximizer of (\ref{eq:SDP-intro}), denoted as $\wh{Z}$, can thus be used as an estimator of the matrix $z^*z^{*\H}$.
The tightness of the SDP (\ref{eq:SDP-intro}) has been thoroughly investigated in the literature of phase synchronization. When $\sigma^2=O(n^{1/2})$, it was proved by \cite{bandeira2017tightness} that the solution to (\ref{eq:SDP-intro}) is a rank-one matrix $\wh{Z}=\wh{z}\wh{z}^{\H}$, with $\wh{z}$ being a global maximizer of (\ref{eq:basic-MLE}). This result was recently proved by \cite{zhong2018near} to hold under a weaker condition $\sigma^2=O\left(\frac{n}{\log n}\right)$. Given the tightness of the SDP and the minimax optimality of the MLE in \cite{gao2020exact}, we can immediately claim that the SDP (\ref{eq:SDP-intro}) is also asymptotically minimax optimal under the condition $\sigma^2=O\left(\frac{n}{\log n}\right)$. Without the condition $\sigma^2=O\left(\frac{n}{\log n}\right)$, whether SDP is still statistically optimal remains as an open question in the literature.

In this paper, we study the statistical properties of the SDP (\ref{eq:SDP-intro}) directly without the need to establish any connection between SDP and MLE. This allows us to go beyond the condition $\sigma^2=O\left(\frac{n}{\log n}\right)$ and we are able to derive sharp statistical error bounds of the SDP (\ref{eq:SDP-intro}) as long as $\sigma^2=o(n)$. According to the minimax lower bound (\ref{eq:lower-intro}), the condition $\sigma^2=o(n)$ is necessary for any estimator to have an error rate of a nontrivial order. To formally state our main result, we introduce a more general statistical estimation setting that allows the possibility of missing data. Instead of observing $Y_{jk}$ for all $1\leq j<k\leq n$, we assume each $Y_{jk}$ is observed with probability $p$. In other words, consider a random graph $A_{jk}\sim \text{Bernoulli}(p)$ independently for all $1\leq j<k\leq n$, and we only observe $Y_{jk}$ that follows (\ref{eq:basic-model}) when $A_{jk}=1$. The SDP can be extended to this more general setting by replacing all $Y_{jk}$'s with $A_{jk}Y_{jk}$'s in (\ref{eq:SDP-intro}). The formula will be given by (\ref{eq:SDP-complex-p}) in Section \ref{sec:main}.

\begin{thm}\label{thm:intro}
Assume $\frac{np}{\sigma^2}\rightarrow\infty$ and $\frac{np}{\log n}\rightarrow\infty$. Let $\wh{Z}$ be a global maximizer of the SDP (\ref{eq:SDP-complex-p}) and $\wh{z}_j=u_j/|u_j|$ for $j\in[n]$ with $u\in\mathbb{C}^n$ being the leading eigenvector of $\wh{Z}$. There exists some $\delta=o(1)$ such that
\begin{eqnarray*}
\frac{1}{n^2}\fnorm{\wh{Z}-z^*z^{*\H}}^2 \leq (1+\delta)\frac{\sigma^2}{np}, \\
\ell(\wh{z},z^*)\leq (1+\delta)\frac{\sigma^2}{2np},
\end{eqnarray*}
with probability at least $1-n^{-8}-\exp\left(-\left(\frac{np}{\sigma^2}\right)^{1/4}\right)$.
\end{thm}

Compared with the minimax lower bound (Theorem \ref{thm:lower} in Section \ref{sec:main}), Theorem \ref{thm:intro} shows that SDP leads to both minimax optimal estimations of the matrix $z^*z^{*\H}$ and of the vector $z^*$. The two error bounds are not just rate-optimal, but the leading constants are sharp as well. We remark that both conditions $\sigma^2=o(np)$ and $\frac{np}{\log n}\rightarrow\infty$ are essential for the results of the above theorem to hold. Since the minimax risk of the problem is of order $\frac{\sigma^2}{np}$, the condition $\sigma^2=o(np)$, which is equivalent to $\frac{\sigma^2}{np}=o(1)$, guarantees that the minimax risk is of smaller order than the trivial one. The order $O(1)$ is trivial, as it can simply be achieved by random guess. The condition $\frac{np}{\log n}\rightarrow\infty$ guarantees that the random graph $A$ is connected with high probability. Our technical analysis would still go through when $p$ is of the same order as $\frac{\log n}{n}$, but in this regime only rate optimality is achieved.

Our analysis of the SDP does not rely on its connection to the MLE, and it is therefore fundamentally different from the approaches considered by \cite{bandeira2017tightness,zhong2018near,ling2020solving}. To study the statistical properties of SDP directly, we consider the following iteration procedure, \footnote{When the denominator (\ref{eq:critical-iter-intro}) is zero, take $V_j^{(t)} = V_j^{(t-1)}$.}
\begin{equation}
V_j^{(t)} = \frac{\sum_{k\in[n]\backslash\{j\}}\bar{Y}_{jk}V_k^{(t-1)}}{\left\|\sum_{k\in[n]\backslash\{j\}}\bar{Y}_{jk}V_k^{(t-1)}\right\|}\in\mathbb{C}^n,\quad j=1,\cdots,n.\label{eq:critical-iter-intro}
\end{equation}
Define the matrix $V^{(t)}\in\mathbb{C}^{n\times n}$ with its $j$th column being $V_j^{(t)}$. The above iteration can be shorthanded as $V^{(t)}=f(V^{(t-1)})$. We use (\ref{eq:critical-iter-intro}) as a non-convex characterization of the SDP (\ref{eq:SDP-intro}), because the solution to (\ref{eq:SDP-intro}) can always be written as $\wh{Z}=\wh{V}^{\H}\wh{V}$ for some $\wh{V}\in\mathbb{C}^{n\times n}$ satisfying the fixed-point equation $\wh{V}=f(\wh{V})$. Note that the iterative procedure (\ref{eq:critical-iter-intro}) resembles the formula of the generalized power method (GPM) \citep{singer2011angular,boumal2016nonconvex,zhong2018near}, \footnote{When the denominator (\ref{eq:GPM-intro}) is zero, take $z_j^{(t)} = z_j^{(t-1)}$.}
\begin{equation}
z_j^{(t)} = \frac{\sum_{k\in[n]\backslash\{j\}}{Y}_{jk}z_k^{(t-1)}}{\left|\sum_{k\in[n]\backslash\{j\}}{Y}_{jk}z_k^{(t-1)}\right|}\in\mathbb{C},\quad j=1,\cdots,n.\label{eq:GPM-intro}
\end{equation}
We can therefore think of (\ref{eq:critical-iter-intro}) as a lift of the GPM (\ref{eq:GPM-intro}) into a higher dimensional space. This allows us to analyze the statistical error of SDP from an iterative algorithm perspective, and previous techniques of analyzing general iterative algorithms in \cite{lu2016statistical,gao2019iterative} can be borrowed for the current purpose.
To understand the exact statistical error of SDP, we establish the following convergence result for the iterative procedure (\ref{eq:critical-iter-intro}),
\begin{equation}
\ell(V^{(t)},z^*)\leq \delta \ell(V^{(t-1)},z^*) + \textsf{optimal statistical error},\quad \text{for all }t\geq 1, \label{eq:important-intro}
\end{equation}
for some $\delta=o(1)$ with high probability, as long as it is properly initialized. Here, with slight abuse of notation, the loss of $\wh{V}$ is defined by
\begin{equation}
\ell(\wh{V},z^*)=\min_{a\in\mathbb{C}^n:\|a\|^2=1}\frac{1}{n}\sum_{j=1}^n\|\wh{V}_j-\bar{z}^*_ja\|^2, \label{eq:loss-V-z}
\end{equation}
which is natural given that the matrix $\wh{Z}=\wh{V}^{\H}\wh{V}$ is used to estimate $z^*z^{*\H}$. Since the SDP solution is a fixed point of the iteration (\ref{eq:critical-iter-intro}), the convergence result (\ref{eq:important-intro}) directly leads to the sharp statistical error bounds in Theorem \ref{thm:intro}.

Our analysis of SDP through (\ref{eq:critical-iter-intro}) also unifies the understandings of the GPM and the MLE. Given the relation between (\ref{eq:critical-iter-intro}) and (\ref{eq:GPM-intro}), the convergence result (\ref{eq:important-intro}) directly implies
\begin{equation}
\ell(z^{(t)},z^*)\leq \delta \ell(z^{(t-1)},z^*) + \textsf{optimal statistical error},\quad \text{for all }t\geq 1, \label{eq:GPM-conv-intro}
\end{equation}
for some $\delta=o(1)$ with high probability, as long as the GPM is properly initialized. This provides an alternative proof to the minimax optimality of the GPM that has been previously established by \cite{gao2020exact}. In addition, just as the SDP can be viewed as a fixed point of the iteration (\ref{eq:critical-iter-intro}), the MLE can be viewed as a fixed point of the iteration (\ref{eq:GPM-intro}). The minimax optimality of the MLE can also be derived. To summarize, we are able to show the exact minimax optimality of SDP, GPM, and MLE using a single proof based on the iterative procedure (\ref{eq:critical-iter-intro}).

In addition to phase synchronization, we also establish the optimality of the SDP for $\mathbb{Z}_2$ synchronization. In the setting of $\mathbb{Z}_2$ synchronization, one observes $Y_{jk}=z_j^*z_k^*+\sigma W_{jk}\in\mathbb{R}$ for $1\leq j<k\leq n$, and the goal is to estimate $z_1^*,\cdots,z_n^*\in\{-1,1\}$. Assume $W_{jk}\sim\n(0,1)$ and each $Y_{jk}$ is observed with probability $p$, we show that the SDP for $\mathbb{Z}_2$ synchronization achieves the error
\begin{equation}
\exp\left(-(1-o(1))\frac{np}{2\sigma^2}\right). \label{eq:exp-rate}
\end{equation}
We also prove a matching lower bound for this problem. Since $\mathbb{Z}_2$ synchronization is a discrete parameter estimation problem, the minimax risk is an exponential function of the signal-to-noise ratio, compared with the polynomial function for phase synchronization. Despite being a continuous optimization method, the SDP is able to adapt to the discreteness of the problem. The exponential rate (\ref{eq:exp-rate}) has been previously derived for $p=1$ by \cite{fei2020achieving}. Our analysis based on the iterative algorithm perspective generalizes their result to more general values of $p\gg\frac{\log n}{n}$.

\paragraph{Paper Organization.} The rest of the paper is organized as follows. In Section \ref{sec:main}, we establish the statistical optimality of the SDP for phase synchronization. The implications of the SDP analysis on the statistical error bounds of GPM and MLE are discussed in Section \ref{sec:GPM-MLE}. The analysis of the SDP for $\mathbb{Z}_2$ synchronization is presented in Section \ref{sec:z2}. Finally, Section \ref{sec:pf} collects all the technical proofs of the paper.

\paragraph{Notation.}

For $d \in \mathbb{N}$, we write $[d] = \{1,\dotsc,d\}$.  Given $a,b\in\mathbb{R}$, we write $a\vee b=\max(a,b)$ and $a\wedge b=\min(a,b)$. For a set $S$, we use $\mathbb{I}\{S\}$ and $|S|$ to denote its indicator function and cardinality respectively. For a complex number $x\in\mathbb{C}$, we use $\bar{x}$ for its complex conjugate, $\re(x)$ for its real part, $\im(x)$ for its imaginary part, and $\abs{x}$ for its modulus. For a complex vector $x\in\mathbb{C}^d$, we use $\norm{x} = \sqrt{\sum_{j=1}^d |x_j|^2}$ for its norm.
For a matrix $B =(B_{jk})\in\mathbb{C}^{d_1\times d_2}$, we use $B^\H \in\mathbb{C}^{d_2\times d_1}$ for its conjugate transpose such that $B^{\H}= (\bar{B}_{kj})$. 
The Frobenius norm and operator norm of $B$ are defined by $\fnorm{B}=\sqrt{\sum_{j=1}^{d_1}\sum_{k=1}^{d_2}|B_{jk}|^2}$ and $\opnorm{B} = \sup_{u\in\mathbb{C}^{d_1},v\in\mathbb{C}^{d_2}:\norm{u}=\norm{v}=1} u^\H Bv$. We use $\Tr(B)$ for the trace of a squared matrix $B$. For $U,V\in\mathbb{C}^{d_1\times d_2}$, $U\circ V\in\mathbb{R}^{d_1\times d_2}$ is the Hadamard product $U\circ V=(U_{jk}V_{jk})$.
The notation $\mathbb{P}$ and $\mathbb{E}$ are generic probability and expectation operators whose distribution is determined from the context. 
For two positive sequences $\{a_n\}$ and $\{b_n\}$, $a_n\lesssim b_n$ or $a_n=O(b_n)$ means $a_n\leq Cb_n$ for some constant $C>0$ independent of $n$. We also write $a_n=o(b_n)$ or $\frac{b_n}{a_n}\rightarrow\infty$ when $\limsup_n\frac{a_n}{b_n}=0$.

\section{Main Results}\label{sec:main}

\subsection{Problem Settings}

Recall that we observe a random graph $A_{jk}\sim\text{Bernoulli}(p)$ independently for all $1\leq j<k\leq n$. For each pair $(j,k)$, we observe $Y_{jk}=z_j^*\bar{z}_k^*+\sigma W_{jk}$ with $W_{jk}\sim\cn(0,1)$ whenever $A_{jk}=1$. The observations can be organized as an adjacency matrix $A$ and a masked version of the pairwise interactions $A\circ Y$. All the matrices $A$, $W$, and $Y$ are Hermitian as we define  $A_{jk}=A_{kj}$, $W_{jk}=\bar{W}_{kj}$, and $Y_{jk}=\bar{Y}_{kj}$  for all $1\leq k<j\leq n$ and $A_{jj} =W_{jj}=0$ and  $Y_{jj}=1$ for all $j\in[n]$. Hence we have the matrix representation $Y=z^*z^{*\H}+\sigma W$.

To estimate the vector $z^*\in\mathbb{C}_1^n$, the MLE is defined as a global maximizer of the following optimization problem
\begin{equation}
\max_{z\in\mathbb{C}_1^n}z^{\H}(A\circ Y)z. \label{eq:MLE-p}
\end{equation}
Since (\ref{eq:MLE-p}) is computationally infeasible, we consider the following convex relaxation of (\ref{eq:MLE-p}) via SDP,
\begin{equation}
\max_{Z=Z^{\H}\in\mathbb{C}^{n\times n}}\Tr((A\circ Y)Z)\quad\text{subject to}\quad \diag(Z)=I_n\text{ and }Z\succeq 0. \label{eq:SDP-complex-p}
\end{equation}
The goal of our paper is to establish the statistical optimality of the SDP (\ref{eq:SDP-complex-p}). We first provide a minimax lower bound as the benchmark of the problem.

\begin{thm}[Theorem 4.1 of \cite{gao2020exact}]\label{thm:lower}
Assume $\sigma^2=o(np)$. Then, we have
\begin{eqnarray*}
\inf_{\wh{Z}\in\mathbb{C}^{n\times n}}\sup_{z\in\mathbb{C}_1^n}\mathbb{E}_z\frac{1}{n^2}\fnorm{\wh{Z}-zz^{\H}}^2\geq (1-\delta)\frac{\sigma^2}{np}, \\
\inf_{\wh{z}\in\mathbb{C}_1^n}\sup_{z\in\mathbb{C}_1^n}\mathbb{E}_z\ell(\wh{z},z)\geq (1-\delta)\frac{\sigma^2}{2np},
\end{eqnarray*}
for some $\delta=o(1)$.
\end{thm}
The above theorem has been established by \cite{gao2020exact} as the minimax lower bound for phase synchronization. In fact, Theorem 4.1 of \cite{gao2020exact} only states the lower bound result for the loss function $\ell(\wh{z},z)$. However, the proof of Theorem 4.1 of \cite{gao2020exact} actually established the lower bound under the loss $\frac{1}{n^2}\fnorm{\wh{Z}-zz^{\H}}^2$, and the lower bound for $\ell(\wh{z},z)$ is proved as a direct consequence  in view of the inequality
$$\inf_{\wh{z}\in\mathbb{C}_1^n}\sup_{z\in\mathbb{C}_1^n}\mathbb{E}_z\ell(\wh{z},z)\geq \frac{1}{2}\inf_{\wh{Z}\in\mathbb{C}^{n\times n}}\sup_{z\in\mathbb{C}_1^n}\mathbb{E}_z\frac{1}{n^2}\fnorm{\wh{Z}-zz^{\H}}^2.$$
Since the solution of the SDP (\ref{eq:SDP-complex-p}) is a matrix, it is natural to study the statistical error under $\frac{1}{n^2}\fnorm{\wh{Z}-zz^{\H}}^2$ in addition to the loss $\ell(\wh{z},z)$.


\subsection{A Convergence Lemma}

Our analysis of the SDP (\ref{eq:SDP-complex-p}) relies on an equivalent non-convex characterization. Since $Z$ is a positive semi-definite Hermitian matrix, it admits a decomposition
$$Z=V^{\H}V,$$
for some $V\in\mathbb{C}^{n\times n}$. Let $V_j$ be the $j$th column of $V$, and we have $Z_{jk}=V_j^{\H}V_k$. In particular, the constraint $\diag(Z)=I_n$ can be written as $Z_{jj}=\|V_j\|^2=1$ for all $j\in[n]$. Replacing $Z$ by $V^{\H}V$, the SDP (\ref{eq:SDP-complex-p}) can be equivalently represented as
\begin{equation}
\max_{V\in\mathbb{C}^{n\times n}}\Tr((A\circ Y)V^{\H}V)\quad\text{subject to}\quad \|V_j\|^2=1\text{ for all }j\in[n]. \label{eq:SDP-nonconvex}
\end{equation}
The formulation (\ref{eq:SDP-nonconvex}) is closely related to the  Burer-Monteiro problem \cite{burer2003nonlinear, javanmard2016phase} for the SDP except that here $V$ is still an $n\times n$ matrix  without dimension reduction. This non-convex formulation allows us to derive sharp statistical error bounds of the SDP (\ref{eq:SDP-complex-p}).

We analyze (\ref{eq:SDP-nonconvex}) through the following iteration procedure,
\begin{equation}
V_j^{(t)} = \begin{cases}
 \frac{\sum_{k\in[n]\backslash\{j\}}A_{jk}\bar{Y}_{jk}V_k^{(t-1)}}{\left\|\sum_{k\in[n]\backslash\{j\}}A_{jk}\bar{Y}_{jk}V_k^{(t-1)}\right\|}, &\sum_{k\in[n]\backslash\{j\}}A_{jk}\bar{Y}_{jk}V_k^{(t-1)} \neq 0,\\
 V_j^{(t-1)}, &\sum_{k\in[n]\backslash\{j\}}A_{jk}\bar{Y}_{jk}V_k^{(t-1)} = 0.
\end{cases} 
 \label{eq:critical-iter}
\end{equation}
Let us shorthand the above formula by
\begin{equation}\label{eqn:f_update}
V^{(t)}=f(V^{(t-1)}),
\end{equation}
by introducing a map $f:\mathbb{C}_1^{n\times n}\rightarrow \mathbb{C}_1^{n\times n}$ such that the $j$th column of $f(V^{(t-1)})$ is given by (\ref{eq:critical-iter}). We use the notation $\mathbb{C}_1^{n\times n}$ for the set of $n\times n$ complex matrices whose columns all have unit norms. The update (\ref{eqn:f_update}) can be seen as a local approach (or more precisely, a block coordinate ascent approach) \cite{wang2017mixing, erdogdu2018convergence} to solve (\ref{eq:SDP-nonconvex}). To see why this is true, consider the following local optimization problem
\begin{equation}
\max_{V_j\in\mathbb{C}^n:\|V_j\|^2=1}\left[V_j^{\H}\left(\sum_{k\in[n]\backslash\{j\}}A_{jk}\bar{Y}_{jk}{V}^{(t-1)}_k\right)+\left(\sum_{k\in[n]\backslash\{j\}}A_{jk}\bar{Y}_{jk}{V}^{(t-1)}_k\right)^{\H}V_j\right].\label{eq:SDP-local}
\end{equation}
The objective of (\ref{eq:SDP-local}) collects the terms in the expansion of $\Tr((A\circ Y)V^{\H}V)=\sum_{jk}A_{jk}\bar{Y}_{jk}V_j^{\H}V_k$ that depend on $V_j$ and replaces  $V_k$ by ${V}^{(t-1)}_k$ for all $k\in[n]\backslash \{j\}$. By simple algebra, we can see the solution of (\ref{eq:SDP-local}) is exactly (\ref{eq:critical-iter}).

Let $\wh{V}$ be a global maximizer of (\ref{eq:SDP-nonconvex}). The matrix $\wh{V}$ must be a fixed point of the map $f$,
\begin{equation}
\wh{V}=f(\wh{V}). \label{eq:fixed-point}
\end{equation}
To see why (\ref{eq:fixed-point}) holds, we consider the local optimization problem (\ref{eq:SDP-local}) with ${V}^{(t-1)}_k$ replaced by $\wh{V}_k$ for all $k\in[n]\backslash \{j\}$.
Thus, as long as $\wh{V}$ maximizes (\ref{eq:SDP-nonconvex}), its $j$th column $\wh{V}_j$ must maximize this local optimization problem, which then implies the fixed-point equation (\ref{eq:fixed-point}).

Since the SDP solution $\wh{Z}=\wh{V}^{\H}\wh{V}$ is an estimator of the matrix $z^{*}z^{*\H}$, we can think of $\wh{V}_j$ as an estimator of $\bar{z}_j^*$ embedded in $\mathbb{C}^n$. Note that
$$z_j^*\bar{z}_k^*=z_j^*a^{\H}a\bar{z}_k^*,$$
for any $a\in\mathbb{C}^n$ such that $\|a\|^2=1$, and thus we can embed each $\bar{z}_j^*$ in $\mathbb{C}^n$ by considering the vector $\bar{z}_j^*a\in\mathbb{C}^n$. This motivates the definition of the loss function $\ell(\wh{V},z^*)$ given in (\ref{eq:loss-V-z}). The following lemma characterizes the evolution of this loss function through the map $f$.

\begin{lemma}\label{lem:critical}
Assume $\frac{np}{\sigma^2}>c_1$ and $\frac{np}{\log n}>c_2$ for some sufficiently large constants $c_1, c_2>0$. Then, for any $\gamma\in[0,1/16)$, we have 
\begin{eqnarray*}
&& \mathbb{P}\left(\ell(f(V),z^*)\leq \delta_1\ell(V,z^*)+(1+\delta_2)\frac{\sigma^2}{2np}\text{ for all }V\in\mathbb{C}_1^{n\times n}\text{ such that }\ell(V,z^*)\leq \gamma\right) \\
&\geq& 1-(2n)^{-1}-\exp\left(-\left(\frac{np}{\sigma^2}\right)^{1/4}\right).
\end{eqnarray*}
where  $\delta_1=C_1\sqrt{\frac{\log n+\sigma^2}{np}}$ and $\delta_2=C_2\left(\gamma^2 + \frac{\log n+\sigma^2}{np}\right)^{1/4}$ for some constants $C_1,C_2>0$.
\end{lemma}

The lemma shows that for any $V\in\mathbb{C}_1^{n\times n}$ that has a nontrivial error, the matrix $f(V)$ will have an error that is smaller by a multiplicative factor $\delta_1$ up to an additive term $(1+\delta_2)\frac{\sigma^2}{2np}$. Define $V^*\in\mathbb{C}_1^{n\times n}$ with the $j$th column given by $V_j^*=\bar{z}_j^*a$ for some $a\in\mathbb{C}^n$ that satisfies $\|a\|^2=1$. We immediately have
$$\ell(f(V^*),z^*)\leq (1+\delta_2)\frac{\sigma^2}{2np}.$$
Thus, the additive term $(1+\delta_2)\frac{\sigma^2}{2np}$ can be understood as the oracle statistical error given the knowledge of $z^*$.

The two conditions $\frac{np}{\sigma^2}\geq c_1$ and $\frac{np}{\log n}\geq c_2$ are essential for the result to hold. While $\frac{np}{\sigma^2}\geq c_1$ makes sure that the statistical error $\frac{\sigma^2}{2np}$ is of a nontrivial order, the condition $\frac{np}{\log n}\geq c_2$ guarantees that the random graph is connected. We can slightly strengthen both conditions to $\frac{np}{\sigma^2}\rightarrow\infty$ and $\frac{np}{\log n}\rightarrow\infty$ so that both $\delta_1$ and $\delta_2$ are varnishing.

\subsection{Statistical Optimality of SDP}

In this section, we show the result of Lemma \ref{lem:critical} implies the statistical optimality of the SDP (\ref{eq:SDP-complex-p}). Since the solution of the SDP can be written as $\wh{Z}=\wh{V}^{\H}\wh{V}$ with $\wh{V}$ satisfying the fixed-point equation (\ref{eq:fixed-point}), we can apply the result of Lemma \ref{lem:critical} to $\wh{V}=f(\wh{V})$ as long as a crude bound $\ell(\wh{V},z^*)\leq \gamma$ can be proved for some $\gamma <1/16$.

\begin{lemma}\label{lem:SDP-crude}
Assume $\frac{np}{\log n}>c$ for some sufficiently large constant $c>0$. Let $\wh{Z}=\wh{V}^{\H}\wh{V}$ be a global maximizer of the SDP (\ref{eq:SDP-complex-p}). Then, there exits some constant $C>0$ such that
$$\ell(\wh{V},z^*)\leq C\sqrt{\frac{\sigma^2+1}{np}},$$
with probability at least $1-n^{-9}$.
\end{lemma}

Under the condition that $\frac{np}{\sigma^2}$ and $\frac{np}{\log n}$ are sufficiently large, we have $\ell(\wh{V},z^*)\leq \gamma$ for some $\gamma<1/16$. Thus, Lemma \ref{lem:critical} and the fact $\wh{V}=f(\wh{V})$ imply that
\begin{equation}
\ell(\wh{V},z^*)\leq \delta_1\ell(\wh{V},z^*)+(1+\delta_2)\frac{\sigma^2}{2np}.\label{eq:apply-l-2-SDP}
\end{equation}
After rearrangement, we obtain the bound $\ell(\wh{V},z^*)\leq \frac{1+\delta_2}{1-\delta_1}\frac{\sigma^2}{2np}$. The result is summarized into the following theorem.
\begin{thm}\label{thm:main1}
Assume $\frac{np}{\sigma^2}>c_1$ and $\frac{np}{\log n}>c_2$ for some sufficiently large constants $c_1, c_2>0$. Let $\wh{Z}=\wh{V}^{\H}\wh{V}$ be a global maximizer of the SDP (\ref{eq:SDP-complex-p}). Then, there exists some  $\delta=C\left(\frac{\log n+\sigma^2}{np}\right)^{1/4}$ for some constant $C>0$, such that
\begin{eqnarray*}
\ell(\wh{V},z^*) &\leq& (1+\delta)\frac{\sigma^2}{2np}, \\
\frac{1}{n^2}\fnorm{\wh{Z}-z^*z^{*\H}}^2 &\leq& (1+\delta)\frac{\sigma^2}{np},
\end{eqnarray*}
with probability at least $1-2n^{-9}-\exp\left(-\left(\frac{np}{\sigma^2}\right)^{1/4}\right)$.
\end{thm}

Theorem \ref{thm:main1} gives sharp statistical error bounds for both loss functions $\ell(\wh{V},z^*)$ and $\frac{1}{n^2}\fnorm{\wh{Z}-z^*z^{*\H}}^2$. While the result for $\ell(\wh{V},z^*)$ is derived from (\ref{eq:apply-l-2-SDP}), the result for $\frac{1}{n^2}\fnorm{\wh{Z}-z^*z^{*\H}}^2$ is a consequence of the inequality
$$\frac{1}{n^2}\fnorm{\wh{V}^{\H}\wh{V}-z^*z^{*\H}}^2\leq 2\ell(\wh{V},z^*),$$
which is established by Lemma \ref{lem:loss-relation} in Section \ref{sec:aux}. Compared with the minimax lower bound in Theorem \ref{thm:lower}, we can conclude that the SDP (\ref{eq:SDP-complex-p}) is minimax optimal for the estimation of the matrix $z^*z^{*\H}$. It not only achieves the optimal rate, but the leading constant is also sharp when $\sigma^2=o(np)$  and $\frac{np}{\log n}\rightarrow\infty$. 
Figure \ref{fig:1} verifies the correctness of the leading constants of the two loss functions. Both loss functions are approximately linear at least when $\sigma^2$ is small. 
When $\sigma^2$ and $np$ are of the same order, SDP does not have the optimal constant anymore, and its asymptotics is predicted by a very different technique \citep{javanmard2016phase}. We also remark that the $\ell_2$ error control does not imply that each individual $z_j^*$ can be accurately recovered. This is also reflected in Figure \ref{fig:1} with the comparison between $\ell_2$ and $\ell_{\infty}$ loss.
\begin{figure}[h]
	\centering
	\includegraphics[width=0.5\textwidth, trim={20 18 20 40},clip]{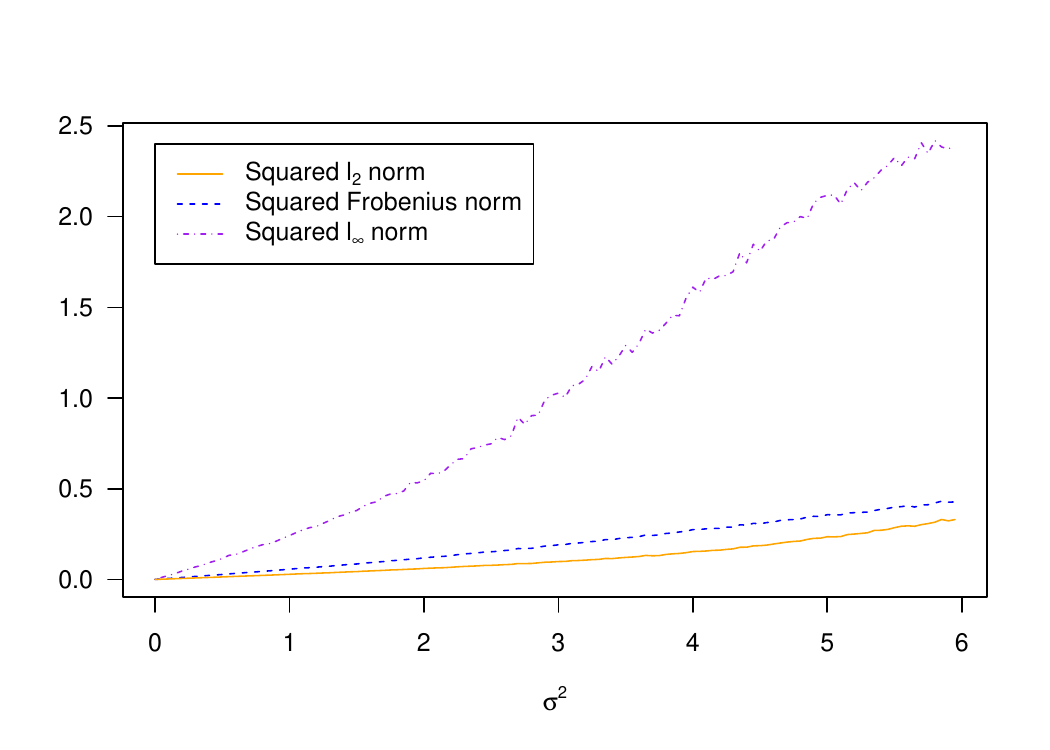}\includegraphics[width=0.5\textwidth, trim={20 18 20 40},clip]{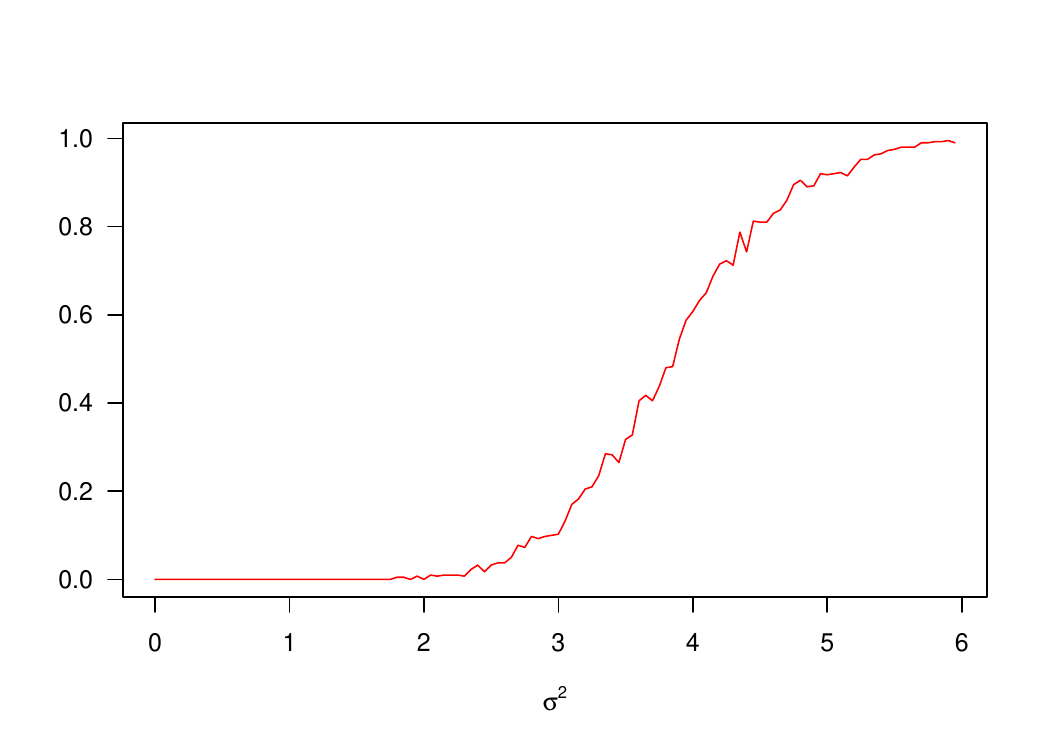}
	\caption{\textsl{Left: average values of $\ell(\wh{V},z^*)$ (orange solid), $\frac{1}{n^2}\fnorm{\wh{Z}-z^*z^{*\H}}^2$ (blue dashed) and $\max_{1\leq j\leq n}\|\wh{V}_j-\bar{z}_j^*a\|^2$ (purple dash-dotted) with $n=100$, $p=0.2$ and $\sigma^2$ varying in $[0,6]$ across $400$ independent experiments. The vector $a\in\mathbb{C}^n$ used in the squared $\ell_{\infty}$ loss is the minimizer of the right hand side of (\ref{eq:loss-V-z}). Right: probability of the event that the second eigenvalue of $\wh{Z}$ is nonzero (i.e., $\wh{Z}$ is not rank-one) for the same experiments.}}
	\label{fig:1}
\end{figure}

We emphasize that our proof of the optimality of the SDP is based on a direct statistical error analysis, regardless of whether the SDP relaxation is tight or not. It is shown by \cite{zhong2018near} that the tightness of SDP (when the solution has rank one) requires $\sigma^2=O\left(\frac{n}{\log n}\right)$ at least when $p=1$. When $\sigma^2=o(np)$, it is possible that SDP is not tight but still statistically optimal. This point is also illustrated by Figure \ref{fig:1}. 

Since the solution of the SDP is a matrix, some post-processing step is required to obtain a vector estimator for $z^*$. This can easily be done by extracting the leading eigenvector of $\wh{Z}$. Let $u\in\mathbb{C}^n$ be the leading eigenvector of $\wh{Z}$, and define $\wh{z}$ with each entry $\wh{z}_j=u_j/|u_j|$. If $u_j=0$ we can take $\wh{z}_j=1$. The statistical optimality of $\wh{z}$ is established by the following result. Recall that for two vectors in $\mathbb{C}_1^n$, the definition of the loss $\ell(\wh{z},z^*)$ is given by (\ref{eq:loss-basic-z}).
\begin{thm}\label{thm:round-phase}
Assume $\frac{np}{\sigma^2}>c_1$ and $\frac{np}{\log n}>c_2$ for some sufficiently large constants $c_1, c_2>0$. Let $\wh{Z}=\wh{V}^{\H}\wh{V}$ be a global maximizer of the SDP (\ref{eq:SDP-complex-p}). Then, 
$$\ell(\wh{z},z^*)\leq \left(1+C\left(\frac{\log n+\sigma^2}{np}\right)^{1/4}\right)\frac{\sigma^2}{2np},$$
with probability at least $1-2n^{-9}-\exp\left(-\left(\frac{np}{\sigma^2}\right)^{1/4}\right)$ for some constant $C>0$.
\end{thm}
Compared with the minimax lower bound in Theorem \ref{thm:lower}, the SDP (\ref{eq:SDP-complex-p}) is also minimax optimal for the estimation of the vector $z^*$ in phase synchronization. Theorem \ref{thm:main1} and Theorem \ref{thm:round-phase} together establish Theorem \ref{thm:intro}.

\section{Implications on Generalized Power Method and MLE}\label{sec:GPM-MLE}

In this section, we show that the analysis of the SDP through Lemma \ref{lem:critical} also leads to statistical optimality of the generalized power method (GPM) and the maximum likelihood estimator (MLE). We note that it has already been established by \cite{gao2020exact} that both GPM and MLE achieve the optimal error bound $(1+o(1))\frac{\sigma^2}{2np}$ under the loss $\ell(\wh{z},z^*)$. By deriving the same results using the analysis of the SDP, we can unify the three proofs and thus form a coherent understanding of the three different methods.

The iteration of GPM of phase synchronization is
\begin{equation}
z_j^{(t)} = 
\begin{cases}
\frac{\sum_{k\in[n]\backslash\{j\}}A_{jk}Y_{jk}z_k^{(t-1)}}{\left|\sum_{k\in[n]\backslash\{j\}}A_{jk}Y_{jk}z_k^{(t-1)}\right|}, & \sum_{k\in[n]\backslash\{j\}}A_{jk}Y_{jk}z_k^{(t-1)}\neq 0,\\
z_j^{(t-1)}, & \sum_{k\in[n]\backslash\{j\}}A_{jk}Y_{jk}z_k^{(t-1)}=0.
\end{cases}
 \label{eq:GPM}
\end{equation}
The similarity between (\ref{eq:GPM}) and (\ref{eq:critical-iter}) is obvious. To make an explicit connection between the two iteration procedures, we can embed (\ref{eq:GPM}) into the space of (\ref{eq:critical-iter}). Let $e_1\in\mathbb{C}^n$ be the first canonical vector with the first entry $1$ and the remaining entries all $0$. It is easy to check that as long as $V_j^{(t-1)}=\bar{z}_j^{(t-1)}e_1$ for all $j\in[n]$, we also have $V_j^{(t)}=\bar{z}_j^{(t)}e_1$ for all $j\in[n]$. This is because once the columns $V_1^{(t)},\cdots,V_n^{(t)}$ lie in the same one-dimensional subspace for some $t$, the iteration (\ref{eq:critical-iter}) remains in this subspace. Thus, the formula (\ref{eq:critical-iter}) exactly describes the GPM iteration (\ref{eq:GPM}). In addition to the connection between (\ref{eq:GPM}) and (\ref{eq:critical-iter}), the two loss functions $\ell(V,z^*)$ and $\ell(z,z^*)$ are also equivalent. Under the condition that $V_j=\bar{z}_je_1$ for all $j\in[n]$, we have
$$\ell(V,z^*)=\ell(z,z^*).$$
Therefore, Lemma \ref{lem:critical} directly implies that
\begin{equation}
\ell(g(z),z^*)\leq \delta_1\ell(z,z^*)+(1+\delta_2)\frac{\sigma^2}{2np},\label{eq:from-lem-critical}
\end{equation}
uniformly over all $z\in\mathbb{C}_1^n$ such that $\ell(z,z^*)< 1/16$ with high probability. The map $g:\mathbb{C}_1^n\rightarrow\mathbb{C}_1^n$ is defined so that (\ref{eq:GPM}) can be shorthanded by $z^{(t)}=g(z^{(t-1)})$.

From (\ref{eq:from-lem-critical}), we know that as long as $\ell(z^{(t-1)},z^*)\leq \gamma$ for some $\gamma < 1/16$, the next step of power iteration (\ref{eq:GPM}) satisfies
\begin{equation}
\ell(z^{(t)},z^*) \leq \delta_1\ell(z^{(t-1)},z^*) + (1+\delta_2)\frac{\sigma^2}{2np}. \label{eq:inter-lem-key}
\end{equation}
The condition $\ell(z^{(t-1)},z^*)\leq \gamma$ then implies $\ell(z^{(t)},z^*)\leq \delta_1\gamma+(1+\delta_2)\frac{\sigma^2}{2np}$. Given that $\frac{\sigma^2}{2np}$ is sufficiently small, we can always choose $\gamma< 1/16$ that satisfies $\frac{\sigma^2}{2np}\leq \frac{\gamma}{2}$. Therefore, $\ell(z^{(t)},z^*)\leq \gamma$. Thus, a simple induction argument implies that (\ref{eq:inter-lem-key}) holds for all $t\geq 1$ as long as $\ell(z^{(0)},z^*)\leq \gamma$. The one-step iteration bound (\ref{eq:inter-lem-key}) immediately implies the linear convergence
\begin{equation}
\ell(z^{(t)},z^*) \leq \delta_1^t\ell(z^{(0)},z^*) + \frac{1+\delta_2}{1-\delta_1}\frac{\sigma^2}{2np}, \label{eq:inter-lem-geo}
\end{equation}
for all $t\geq 1$. It has been shown by \cite{gao2020exact} that the initial error condition $\ell(z^{(0)},z^*)\leq \gamma< 1/16$ is satisfied by a simple eigenvector method. That is, $z_j^{(0)}=v_j/|v_j|$ with $v\in\mathbb{C}^n$ being the leading eigenvector of the matrix $A\circ Y$. Then, (\ref{eq:inter-lem-geo}) implies $\ell(z^{(t)},z^*)\leq (1+o(1))\frac{\sigma^2}{2np}$ for all $t\geq \log\left(\frac{1}{\sigma^2}\right)$.

The optimality of the MLE can be derived from a similar embedding argument. Let $\wh{z}$ be a global maximizer of (\ref{eq:MLE-p}). By the definition of $\wh{z}$, its $j$th entry must satisfy
$$\wh{z}_j=\argmin_{z_j\in\mathbb{C}_1}\sum_{k\in[n]\backslash\{j\}}A_{jk}|Y_{jk}-z_j\bar{\wh{z}}_k|^2=\frac{\sum_{k\in[n]\backslash\{j\}}A_{jk}Y_{jk}\wh{z}_k}{\left|\sum_{k\in[n]\backslash\{j\}}A_{jk}Y_{jk}\wh{z}_k\right|},$$ 
as long as $\sum_{k\in[n]\backslash\{j\}}A_{jk}Y_{jk}\wh{z}_k \neq 0$.
By letting $\wh{V}=e_1\wh{z}^{\H}$, it can be shown that the fixed-point equation $\wh{V}=f(\wh{V})$ holds. Given the equivalence of the loss $\ell(\wh{V},z^*)=\ell(\wh{z},z^*)$, as long as we can show a crude bound $\ell(\wh{z},z^*)\leq \gamma<1/16$ for the MLE, the inequality (\ref{eq:apply-l-2-SDP}) holds and it can be written as
$$\ell(\wh{z},z^*)\leq \delta_1\ell(\wh{z},z^*)+(1+\delta_2)\frac{\sigma^2}{2np},$$
which implies $\ell(\wh{z},z^*)\leq \frac{1+\delta_2}{1-\delta_1}\frac{\sigma^2}{2np}$ after rearrangement. The crude bound $\ell(\wh{z},z^*)\leq \gamma<1/16$ can be easily established for the MLE using the argument in \cite{gao2020exact} or by a similar argument to the proof of Lemma \ref{lem:SDP-crude}, and thus we obtain the optimal error bound $\ell(\wh{z},z^*)\leq(1+o(1))\frac{\sigma^2}{2np}$ for the MLE.

\section{SDP for $\mathbb{Z}_2$ Synchronization}\label{sec:z2}

In this section, we show our analysis of SDP can also be applied to $\mathbb{Z}_2$ synchronization and leads to a sharp exponential statistical error rate. Suppose we observe a random graph $A_{jk}\sim\text{Bernoulli}(p)$ independently for all $1\leq j<k\leq n$. For each pair $(j,k)$, we observe $Y_{jk}=z_j^*z_k^*+\sigma W_{jk}$ with $z_j^*,z_k^*\in\{-1,1\}$ and $W_{jk}\sim\n(0,1)$ whenever $A_{jk}=1$. In $\mathbb{Z}_2$ synchronization, our goal is to estimate the binary vector $z^*\in\{-1,1\}^n$ from observations $\{A_{jk}\}_{1\leq j<k\leq n}$ and $\{A_{jk}Y_{jk}\}_{1\leq j<k\leq n}$. We organize the data into two matrices $A$ and $A\circ Y$. Both the matrices $A$ and $Y$ are symmetric as we define $Y_{jk}=Y_{kj}$ and $A_{jk}=A_{kj}$ for all $1\leq k<j\leq n$ and $Y_{jj}=A_{jj}=0$ for all $j\in[n]$.

With slight abuse of notation, we consider the loss function
$$\ell(\wh{z},z)=\min_{a\in\{-1,1\}}\frac{1}{n}\sum_{j=1}^n|\wh{z}_j-z_ja|^2,$$
for any $\wh{z},z\in\{-1,1\}^n$. Since $|\wh{z}_j-z_ja|^2=4\mathbb{I}\{\wh{z}_j\neq z_ja\}$, the loss $\ell(\wh{z},z)$ is also called the misclassification proportion in a clustering problem \citep{zhang2016minimax,gao2019iterative}. We first present the minimax lower bound of $\mathbb{Z}_2$ synchronization under this loss function.

\begin{thm}\label{thm:lower-z2}
Assume $\frac{np}{\sigma^2}>c_1$ and $\frac{np}{\log n}>c_2$ for some sufficiently large constants $c_1,c_2>0$. Then, we have
\begin{eqnarray*}
\inf_{\wh{Z}\in\mathbb{R}^{n\times n}}\sup_{z\in\{-1,1\}^n}\mathbb{E}_z\frac{1}{n^2}\fnorm{\wh{Z}-zz^{\T}}^2\geq \exp\left(-(1+\delta)\frac{np}{2\sigma^2}\right), \\
\inf_{\wh{z}\in\{-1,1\}^n}\sup_{z\in\{-1,1\}^n}\mathbb{E}_z\ell(\wh{z},z)\geq \exp\left(-(1+\delta)\frac{np}{2\sigma^2}\right),
\end{eqnarray*}
where $\delta=C\sqrt{\frac{\log n + \sigma^2}{np}}$ for some constant $C>0$.
\end{thm}

When $p=1$, the above result has been proved by \cite{fei2020achieving}, but the lower bound result for a general $p$ is unknown in the literature. Compared with Theorem \ref{thm:lower}, the minimax lower bound for $\mathbb{Z}_2$ synchronization is an exponential function of the signal-to-noise ratio, a consequence of the discreteness of the problem.

To estimate $z^*\in\{-1,1\}^n$, the MLE is defined as the global maximizer of the following optimization problem
\begin{equation}
\max_{z\in\{-1,1\}^n}z^{\T}(A\circ Y)z. \label{eq:MLE-z2}
\end{equation}
Similar to (\ref{eq:SDP-complex-p}), a convex relaxation of (\ref{eq:MLE-z2}) leads to the following SDP,
\begin{equation}
\max_{Z=Z^{\T}\in\mathbb{R}^{n\times n}}\Tr((A\circ Y)Z)\quad\text{subject to}\quad \diag(Z)=I_n\text{ and }Z\succeq 0. \label{eq:SDP-real}
\end{equation}
The SDP for $\mathbb{Z}_2$ synchronization is almost in the exact form of (\ref{eq:SDP-complex-p}). The only difference between (\ref{eq:SDP-real}) and (\ref{eq:SDP-complex-p}) is that the optimization (\ref{eq:SDP-real}) is over real symmetric matrices and the optimization of (\ref{eq:SDP-complex-p}) is over complex Hermitian matrices.

Our analysis of the SDP (\ref{eq:SDP-real}) for $\mathbb{Z}_2$ synchronization relies on a non-convex characterization that is similar to (\ref{eq:SDP-nonconvex}). For any $Z$ that is a positive semi-definite real symmetric matrix, it admits a decomposition $Z=V^{\T}V$ for some $V\in\mathbb{R}^{n\times n}$. By writing the $j$th column of $V$ as $V_j$, we can replace the constraint $\diag(Z)=I_n$ by $\|V_j\|^2=1$ for all $j\in[n]$. Then, an equivalent non-convex form of the SDP (\ref{eq:SDP-real}) is
\begin{equation}
\max_{V\in\mathbb{R}^{n\times n}}\Tr((A\circ Y)V^{\T}V)\quad\text{subject to}\quad \|V_j\|^2=1\text{ for all }j\in[n]. \label{eq:SDP-nonconvex-real}
\end{equation}
We will study the solution of (\ref{eq:SDP-nonconvex-real}) using the following loss function,
$$\ell(\wh{V},z)=\min_{a\in\mathbb{R}^n:\|a\|^2=1}\frac{1}{n}\sum_{j=1}^n\|\wh{V}_j-z_ja\|^2.$$
By the same argument that leads to (\ref{eq:fixed-point}), we know that if $\wh{V}$ is a global maximizer of (\ref{eq:SDP-nonconvex-real}), it will satisfy the equation $\wh{V}=f(\wh{V})$, where $f:\mathbb{R}_1^{n\times n}\rightarrow\mathbb{R}_1^{n\times n}$ is a map such that the $j$th column of $f(\wh{V})$ is given by
\begin{align*}
[f(\wh{V})]_j=
\begin{cases}\frac{\sum_{k\in[n]\backslash\{j\}}A_{jk}{Y}_{jk}\wh{V}_k^{(t-1)}}{\left\|\sum_{k\in[n]\backslash\{j\}}A_{jk}{Y}_{jk}\wh{V}_k^{(t-1)}\right\|}, & \sum_{k\in[n]\backslash\{j\}}A_{jk}{Y}_{jk}\wh{V}_k^{(t-1)} \neq 0,\\
\wh{V}_j^{(t-1)}, & \sum_{k\in[n]\backslash\{j\}}A_{jk}{Y}_{jk}\wh{V}_k^{(t-1)}= 0.
\end{cases}
\end{align*}
Here, we use the notation $\mathbb{R}_1^{n\times n}$ for set of $n\times n$ real matrices whose columns all have unit norms.
For each $j\in[n]$, define the random variable
$$U_j=\frac{\sigma}{(n-1)p}\sum_{k\in[n]\backslash\{j\}}z_k^*A_{jk}W_{jk}.$$
The following lemma characterizes the evolution of the loss $\ell(V,z^*)$ through the map $f$.
\begin{lemma}\label{lem:critical-z2}
Assume $\frac{np}{\sigma^2} \geq c_2$ and $\frac{np}{\log n}\geq c_2$ for some sufficiently large constants $c_1,c_2>0$. Then, for any $\gamma\in [0,1/16)$, we have
\begin{eqnarray*}
&& \mathbb{P}\left(\ell(f(V),z^*)\leq \frac{1}{2}\ell(V,z^*)+\frac{4}{n}\sum_{j=1}^n\mathbb{I}\{|U_j|>1-\delta\}\text{ for all }V\in\mathbb{R}_1^{n\times n}\text{ such that }\ell(V,z^*)\leq \gamma\right) \\
&\geq& 1-2n^{-9},
\end{eqnarray*}
where $\delta=C\left(\sqrt{\gamma} + \sqrt{\frac{\log n+\sigma^2}{np}}\right)$ for some constant $C>0$.
\end{lemma}

Lemma \ref{lem:critical-z2} immediately implies that for any $\wh{V}$ that satisfies the fixed-point equation $\wh{V}=f(\wh{V})$ and the crude error bound $\ell(\wh{V},z^*)\leq\gamma <1/16$, we have
\begin{equation}
\ell(\wh{V},z^*) \leq \frac{8}{n}\sum_{j=1}^n\mathbb{I}\{|U_j|>1-\delta\},\label{eq:bounded-by-ind}
\end{equation}
with high probability. The property of the random variable $\frac{8}{n}\sum_{j=1}^n\mathbb{I}\{|U_j|>1-\delta\}$ can be easily analyzed, and we present the following lemma.

\begin{lemma}\label{lem:stat-error-z2}
Assume $\frac{np}{\sigma^2}\geq c_1$ and $\frac{np}{\log n}\geq c_2$ for some sufficiently large constants $c_1,c_2>0$. Then, for any $\delta\in(0,1)$, we have
$$\frac{8}{n}\sum_{j=1}^n\mathbb{I}\{|U_j|>1-\delta\}\leq \exp\left(-(1-\delta')\frac{np}{2\sigma^2}\right),$$
with probability at least $1- \exp\left(-\sqrt{\frac{np}{\sigma^2}}\right) - n^{-9}$, where $\delta'=C\br{\delta +\sqrt{\frac{\log n}{np}}}$ for some constant $C>0$. If we additionally assume $(1-\delta')\frac{np}{2\sigma^2}>\log n$, then
$$\frac{8}{n}\sum_{j=1}^n\mathbb{I}\{|U_j|>1-\delta\}=0,$$
with probability at least $1- \exp\left(-\sqrt{\frac{np}{\sigma^2}}\right) - n^{-9}$. 
\end{lemma}

We also need a lemma to establish a crude error bound for $\ell(\wh{V},z^*)$.
\begin{lemma}\label{lem:SDP-crude-z2}
Assume $\frac{np}{\log n}\geq c$ for some constant $c>0$. Let $\wh{Z}=\wh{V}^{\T}\wh{V}$ be a global maximizer of the SDP (\ref{eq:SDP-real}). Then, there exits some constant $C>0$ such that
$$\ell(\wh{V},z^*)\leq C\sqrt{\frac{\sigma^2+1}{np}},$$
with probability at least $1-n^{-9}$.
\end{lemma}

The results of Lemma \ref{lem:critical-z2}, Lemma \ref{lem:stat-error-z2} and Lemma \ref{lem:SDP-crude-z2} immediately imply the statistical optimality of the SDP (\ref{eq:SDP-real}).
\begin{thm}\label{thm:main-z2}
Assume $\frac{np}{\sigma^2}\geq c_1$ and $\frac{np}{\log n}\geq c_2$ for some sufficiently large constants $c_1,c_2>0$.  Let $\wh{Z}=\wh{V}^{\T}\wh{V}$ be a global maximizer of the SDP (\ref{eq:SDP-real}) and $u\in\mathbb{R}^n$ be the leading eigenvector of $\wh{Z}$. Define $\wh{z}$ with each entry $\wh{z}_j=u_j/|u_j|$. If $u_j=0$ we can take $\wh{z}_j=1$. Then, there exists some $\delta=C\br{\frac{\log n+\sigma^2}{np}}^\frac{1}{4}$ for some constant $C>0$, such that
\begin{eqnarray*}
\ell(\wh{V},z^*) &\leq& \exp\left(-(1-\delta)\frac{np}{2\sigma^2}\right), \\
\frac{1}{n^2}\fnorm{\wh{Z}-z^*z^{*\T}}^2 &\leq& \exp\left(-(1-\delta)\frac{np}{2\sigma^2}\right), \\
\ell(\wh{z},z^*) &\leq& \exp\left(-(1-\delta)\frac{np}{2\sigma^2}\right),
\end{eqnarray*}
with probability at least $1- \exp\left(-\sqrt{\frac{np}{\sigma^2}}\right) - 2n^{-9}$. 
Moreover, if we additionally assume $\sigma^2<(1-\epsilon)\frac{np}{2\log n}$ for some arbitrarily small constant $\epsilon>0$, the SDP solution $\wh{Z}$ is a rank-one matrix that satisfies $\wh{Z}=z^*z^{*\T}$ with probability at least $1- \exp\left(-\sqrt{\frac{np}{\sigma^2}}\right) -  2n^{-9}$.
\end{thm}

While the first conclusion of the theorem is a direct consequence of Lemma \ref{lem:critical-z2}, the second conclusion can be derived from the inequality
$$\frac{1}{n^2}\fnorm{\wh{V}^{\T}\wh{V}-z^*z^{*\T}}^2\leq 2\ell(\wh{V},z^*),$$
which is established by Lemma \ref{lem:loss-relation} in Section \ref{sec:aux}. The result for the loss $\ell(\wh{z},z^*)$ is resulted from a matrix perturbation bound \citep{davis1970rotation}.

Theorem \ref{thm:main-z2} has established the minimax optimality of the SDP (\ref{eq:SDP-real}) for $\mathbb{Z}_2$ synchronization in view of the matching lower bound results in Theorem \ref{thm:lower-z2}. The special case $p=1$ recovers the results of \cite{fei2020achieving}. Moreover, under the condition $\sigma^2<(1-\epsilon)\frac{np}{2\log n}$, we show that the SDP solution $\wh{Z}$ is exactly rank-one and therefore rounding through the leading eigenvector is not needed. This result generalizes the exact recovery threshold of $\mathbb{Z}_2$ synchronization when $p=1$ \citep{bandeira2017tightness,bandeira2018random,abbe2020entrywise}. The phenomenon that SDP can achieve exact recovery has also been revealed in community detection under stochastic block models \citep{hajek2016achieving,hajek2016achieving2,perry2017semidefinite,amini2018semidefinite,chen2018convexified,li2018convex}.

We shall compare Theorem \ref{thm:main-z2} to Theorem \ref{thm:main1} and Theorem \ref{thm:round-phase}. Though the two SDPs (\ref{eq:SDP-real}) and (\ref{eq:SDP-complex-p}) have the same type of constraints, the difference of the domain implies two types of convergence rates $\exp\left(-(1-o(1))\frac{np}{2\sigma^2}\right)$ and $(1+o(1))\frac{\sigma^2}{2np}$. It is quite surprising that the SDP (\ref{eq:SDP-real}), a continuous optimization problem, is able to achieve an exponential rate, which is typical for a discrete problem. The adaptation of the SDP (\ref{eq:SDP-real}) to the discrete structure is a consequence of the fact that both (\ref{eq:SDP-real}) and (\ref{eq:SDP-nonconvex-real}) are optimization problems over $\mathbb{R}^{n\times n}$. We make this effect explicit by bounding the statistical error by the random variable $\frac{8}{n}\sum_{j=1}^n\mathbb{I}\{|U_j|>1-\delta\}$ in Lemma \ref{lem:critical-z2}.

To close this section, we briefly discuss the implications of Lemma \ref{lem:critical-z2} on the MLE (\ref{eq:MLE-z2}) and the generalized power method defined by the iteration procedure
\begin{align}
z_j^{(t)} =
\begin{cases}
 \frac{\sum_{k\in[n]\backslash\{j\}}A_{jk}Y_{jk}z_k^{(t-1)}}{\left|\sum_{k\in[n]\backslash\{j\}}A_{jk}Y_{jk}z_k^{(t-1)}\right|}, & \sum_{k\in[n]\backslash\{j\}}A_{jk}Y_{jk}z_k^{(t-1)}\neq 0,\\
 z_j^{(t-1)}, & \sum_{k\in[n]\backslash\{j\}}A_{jk}Y_{jk}z_k^{(t-1)} =0.
\end{cases}
\label{eq:GPM-z2}
\end{align}
We note that the iteration (\ref{eq:GPM-z2}) is real-valued so that we always have $z_j^{(t)}\in\{-1,1\}$, which makes it different from (\ref{eq:GPM}). The statistical optimality of the generalized power method (\ref{eq:GPM-z2}) has been established by \cite{gao2019iterative} for $\mathbb{Z}_2$ synchronization when $p=1$. Following the same argument in Section \ref{sec:GPM-MLE}, we can embed both MLE and GPM into $\mathbb{R}_1^{n\times n}$, and thus Lemma \ref{lem:critical-z2} also implies that both MLE and GPM achieve the optimal rate $\exp\left(-(1-o(1))\frac{np}{2\sigma^2}\right)$ for a general $p$ as well. Just as what we have for phase synchronization, the analyses of MLE, GPM, and SDP for $\mathbb{Z}_2$ synchronization are all based on Lemma \ref{lem:critical-z2}, and thus we have unified the three different methods from an iterative algorithm perspective.

\section{Proofs}\label{sec:pf}

This section presents the proofs of all technical results in the paper. We first list some auxiliary lemmas in Section \ref{sec:aux}. The key lemmas of the SDP analyses, Lemma \ref{lem:critical} and Lemma \ref{lem:critical-z2}, are proved in Section \ref{sec:pf-cri-lem} and Section \ref{sec:pf-cri-lem-z2}, respectively. We then prove the main results including Theorem \ref{thm:main1}, Theorem \ref{thm:round-phase} and Theorem \ref{thm:main-z2} in Section \ref{sec:pf-main}. Theorem \ref{thm:lower-z2} is proved in Section \ref{sec:pf-lower}. Finally, the proofs of Lemma \ref{lem:SDP-crude}, Lemma \ref{lem:SDP-crude-z2} and Lemma \ref{lem:stat-error-z2} are given in Section \ref{sec:pf-ini}.

\subsection{Some Auxiliary Lemmas}\label{sec:aux}

\begin{lemma}\label{lem:ER-graph}
Assume $\frac{np}{\log n}\rightarrow\infty$. Then, there exists a constant $C>0$, such that
$$\max_{j\in[n]}\left(\sum_{k\in[n]\backslash\{j\}}(A_{jk}-p)\right)^2\leq Cnp\log n,$$
and
$$\opnorm{A-\mathbb{E}A}\leq C\sqrt{np},$$
with probability at least $1-n^{-10}$.
\end{lemma}
\begin{proof}
The first result is a direct application of union bound and Bernstein's inequality. The second result is Theorem 5.2 of \cite{lei2015consistency}. 
\end{proof}

The following result is essentially Corollary 3.11 of \cite{bandeira2016sharp}. The specific form that we need is from Lemma 5.2 in \cite{gao2020exact}.
\begin{lemma}[Corollary 3.11 of \cite{bandeira2016sharp}]\label{lem:bandeira}
Assume $\frac{np}{\log n}\rightarrow\infty$. Then, there exists a constant $C>0$, such that
$$\opnorm{A\circ W}\leq C\sqrt{np},$$
with probability at least $1-n^{-10}$. The result holds for both complex $W$ in Section \ref{sec:main} and real $W$ in Section \ref{sec:z2}. 
\end{lemma}

\begin{lemma}[Lemma 5.3 of \cite{gao2020exact}]\label{lem:talagrand}
Assume $\frac{np}{\log n}>c$ for some sufficiently large constant $c>0$. Consider independent random variables $X_{jk}\sim\n(0,1)$ for $1\leq j<k\leq n$. Assume $X_{kj}=X_{jk}$ for  $1\leq j<k \leq n$. Then, we have
$$\sum_{j=1}^n\left(\sum_{k\in[n]\backslash\{j\}}A_{jk}X_{jk}\right)^2\leq n(n-1)p+C\sqrt{n^3p^2\log n},$$
with probability at least $1-n^{-10}$. The same result holds if $X_{kj}=-X_{jk}$ is assumed instead for  $1\leq j<k \leq n$.
\end{lemma}

\begin{lemma}[Lemma 13 of \cite{gao2016optimal}]\label{lem:gaomalu}
Consider independent random variables $X_j\sim \n(0,1)$ and $E_j\sim\text{Bernoulli}(p)$. Then,
$$\mathbb{P}\left(\left|\sum_{j=1}^nX_jE_j/p\right|>t\right)\leq 2\exp\left(-\min\left(\frac{pt^2}{16n},\frac{pt}{2}\right)\right),$$
for any $t>0$.
\end{lemma}

\begin{lemma}\label{lem:middle}
The following three statements hold:
\begin{enumerate}
\item For any $x,y\in\mathbb{C}^n$ such that $\|y\|=1$ and $\re(y^{\H}x)>0$, we have
$$\left\|\frac{x}{\|x\|}-y\right\|^2\leq \frac{\|(I_n-yy^{\H})x\|^2+|\im(y^{\H}x)|^2}{|\re(y^{\H}x)|^2}.$$
\item For any $x,y\in\mathbb{R}^n$ such that $\|y\|=1$ and $y^{\T}x>0$, we have
$$\left\|\frac{x}{\|x\|}-y\right\|^2\leq \frac{\|(I_n-yy^{\T})x\|^2}{|y^{\T}x|^2}.$$
\item For any $x\in\mathbb{C}$ such that $\re(x)>0$, we have
$$\left|\frac{x}{|x|}-1\right|^2\leq \frac{|\im(x)|^2}{|\re(x)|^2}.$$
\end{enumerate}
\end{lemma}
\begin{proof}
It is easy to see that the last two statements are special cases of the first one. Thus, we only need to prove the first statement. Note that
\begin{eqnarray*}
\left\|\frac{x}{\|x\|}-y\right\|^2 &=& \left\|\frac{(I_n-yy^{\H})x+(y^{\H}x)y}{\|x\|}-y\right\|^2 \\
&=& \frac{\|(I_n-yy^{\H})x\|^2+\left|y^{\H}x-\|x\|\right|^2}{\|x\|^2} \\
&=& \frac{b^2+\left(a-\sqrt{a^2+b^2}\right)^2}{a^2+b^2},
\end{eqnarray*}
where $a=\re(y^{\H}x)>0$ and $b=\sqrt{|\im(y^{\H}x)|^2+\|(I_n-yy^{\H})x\|^2}$. Since
$$\frac{b^2+\left(a-\sqrt{a^2+b^2}\right)^2}{a^2+b^2}=\frac{2b^2}{a^2+b^2+a\sqrt{a^2+b^2}}\leq \frac{b^2}{a^2},$$
the proof is complete.
\end{proof}

\begin{lemma}\label{lem:loss-relation}
For any $V=(V_1,\cdots, V_n)\in\mathbb{C}^{n\times n}$ and any $z\in\mathbb{C}_1^n$ such that $\|V_j\|=1$ for all $j\in[n]$, we have
$$n^{-2}\fnorm{{V}^{\H}V-zz^{\H}}^2\leq 2\ell(V,z).$$
For any $V=(V_1,\cdots, V_n)\in\mathbb{R}^{n\times n}$ and any $z\in\{-1,1\}^n$ such that $\|V_j\|=1$ for all $j\in[n]$, we have
$$n^{-2}\fnorm{{V}^{\T}V-zz^{\T}}^2\leq 2\ell(V,z).$$
\end{lemma}
\begin{proof}
We only prove the complex version of the inequality. The real version follows the same argument. By definition, we have
\begin{eqnarray*}
\ell(V,z) &=& \left(2-\max_{a\in\mathbb{C}^n:\|a\|^2=1}\left(a^{\H}\left(\frac{1}{n}\sum_{j=1}^n{z}_jV_j\right)+\left(\frac{1}{n}\sum_{j=1}^n{z}_jV_j\right)^{\H}a\right)\right) \\
&=& 2\left(1-\left\|\frac{1}{n}\sum_{j=1}^n{z}_jV_j\right\|\right).
\end{eqnarray*}
Then,
\begin{eqnarray*}
n^{-2}\fnorm{{V}^{\H}V-zz^{\H}}^2 &=& \frac{1}{n^2}\sum_{j=1}^n\sum_{l=1}^n|{V}_j^{\H}{V}_l-z_j\bar{z}_l|^2 \\
& \leq  & \frac{1}{n^2}\sum_{j=1}^n\sum_{l=1}^n\left(2-{V}_j^{\H}{V}_l\bar{z}_jz_l-{V}_l^{\H}{V}_jz_j\bar{z}_l\right) \\
&=& 2\left(1-\left\|\frac{1}{n}\sum_{j=1}^n{z}_jV_j\right\|^2\right).
\end{eqnarray*}
Therefore, $n^{-2}\fnorm{{V}^{\H}V-zz^{\H}}^2\leq \ell(V,z)\left(2-\frac{1}{2}\ell(V,z)\right)\leq 2\ell(V,z)$, and the proof is complete.
\end{proof}

\subsection{Proof of Lemma \ref{lem:critical}}\label{sec:pf-cri-lem}

We organize the proof into four steps. We first list a few high-probability events in Step 1. These events are assumed to be true in later steps. Step 2 provides an error decomposition of $\ell(f(V),z^*)$, and then each error term in the decomposition will be analyzed and bounded in Step 3. Finally, we combine the bounds and derive the desired result in Step 4.
\paragraph{Step 1: Some high-probability events.} By Lemma \ref{lem:ER-graph}, Lemma \ref{lem:bandeira}, and Lemma \ref{lem:talagrand}, we know that 
\begin{eqnarray}
\label{eq:hollow1} \min_{j\in[n]}\sum_{k\in[n]\backslash\{j\}}A_{jk} &\geq& (n-1)p - C\sqrt{np\log n}, \\
\label{eq:hollow2}\max_{j\in[n]}\sum_{k\in[n]\backslash\{j\}}A_{jk} &\leq& (n-1)p + C\sqrt{np\log n}, \\
\label{eq:hollow3}\opnorm{A-\mathbb{E}A} &\leq& C\sqrt{np}, \\
\label{eq:hollow4}\opnorm{A\circ W} &\leq& C\sqrt{np}, \\
\label{eq:hollow6} \sum_{j=1}^n\left|\sum_{k\in[n]\backslash\{j\}}A_{jk}\im(\bar{W}_{jk}\bar{z}_k^*z_j^*)\right|^2 &\leq&  \frac{n^2p}{2}\left(1+C\sqrt{\frac{\log n}{n}}\right), \\
 \label{eq:hollow7} \sum_{j=1}^n\left|\sum_{k\in[n]\backslash\{j\}}A_{jk}\re(\bar{W}_{jk}\bar{z}_k^*z_j^*)\right|^2 &\leq&  \frac{n^2p}{2}\left(1+C\sqrt{\frac{\log n}{n}}\right),
\end{eqnarray}
all hold with probability at least $1-n^{-9}$ for some constant $C>0$. To establish (\ref{eq:hollow6})-(\ref{eq:hollow7}), note that $\sqrt{2}\im(\bar{W}_{jk}\bar{z}_k^*z_j^*), \sqrt{2}\re(\bar{W}_{jk}\bar{z}_k^*z_j^*)$ are all independently standard normally distributed for $1\leq j<k\leq n$. We also have $-\im(\bar{W}_{jk}\bar{z}_k^*z_j^*) =\im(W_{jk}z_k^*\bar{z}_j^*)=\im(\bar{W}_{kj}\bar{z}_j^*z_k^*)$ and $\re(\bar{W}_{jk}\bar{z}_k^*z_j^*) =\re(W_{jk}z_k^*\bar{z}_j^*)=\re(\bar{W}_{kj}\bar{z}_j^*z_k^*)$ for any $1\leq j<k\leq n$.

In addition to (\ref{eq:hollow1})-(\ref{eq:hollow7}), we need another high-probability inequality. For a sufficiently small $\rho$ such that  $\frac{\rho^2 np}{\sigma^2}$ is sufficiently large, we want to upper bound the random variable $\sum_{j=1}^n\mathbb{I}\left\{\frac{2\sigma}{np}\left|\sum_{k\in[n]\backslash\{j\}}A_{jk}\bar{W}_{jk}\bar{z}_k^*\right|>\rho\right\}$. The existence of such $\rho$ is guaranteed by the condition $\frac{np}{\sigma^2}$ is sufficiently large, and the specific choice will be given later. We first bound its expectation by Lemma \ref{lem:gaomalu},
\begin{eqnarray*}
\sum_{j=1}^n\mathbb{P}\left\{\frac{2\sigma}{np}\left|\sum_{k\in[n]\backslash\{j\}}A_{jk}\bar{W}_{jk}\bar{z}_k^*\right|>\rho\right\} &\leq& \sum_{j=1}^n\mathbb{P}\left\{\frac{2\sigma}{np}\left|\sum_{k\in[n]\backslash\{j\}}A_{jk}\re(\bar{W}_{jk}\bar{z}_k^*)\right|>\frac{\rho}{2}\right\} \\
&& + \sum_{j=1}^n\mathbb{P}\left\{\frac{2\sigma}{np}\left|\sum_{k\in[n]\backslash\{j\}}A_{jk}\im(\bar{W}_{jk}\bar{z}_k^*)\right|>\frac{\rho}{2}\right\} \\
&\leq& 4n\exp\left(-\frac{\rho^2 np}{256\sigma^2}\right) + 4n\exp\left(-\frac{\rho np}{8\sigma}\right).
\end{eqnarray*}
By Markov inequality, we have
\begin{equation}
\sum_{j=1}^n\mathbb{I}\left\{\frac{2\sigma}{np}\left|\sum_{k\in[n]\backslash\{j\}}A_{jk}\bar{W}_{jk}\bar{z}_k^*\right|>\rho\right\} \leq \frac{4\sigma^2}{\rho^2p} \exp\left(-\frac{1}{16}\sqrt{\frac{\rho^2 np}{\sigma^2}}\right), \label{eq:hollow5}
\end{equation}
with probability at least
\begin{eqnarray*}
&& 1-\frac{\rho^2pn}{\sigma^2}\left(\exp\left(-\frac{\rho^2 np}{256\sigma^2}+\frac{1}{16}\sqrt{\frac{\rho^2 np}{\sigma^2}}\right) + \exp\left(-\frac{\rho np}{8\sigma}+\frac{1}{16}\sqrt{\frac{\rho^2 np}{\sigma^2}}\right)\right) \\
&\geq& 1-\frac{2\rho^2pn}{\sigma^2}\exp\left(-\frac{1}{16}\sqrt{\frac{\rho^2 np}{\sigma^2}}\right) \\
&\geq& 1-\exp\left(-\frac{1}{32}\sqrt{\frac{\rho^2 np}{\sigma^2}}\right).
\end{eqnarray*}
Finally, we conclude that the events (\ref{eq:hollow1})-(\ref{eq:hollow5}) hold simultaneously with probability at least $1-n^{-9}-\exp\left(-\frac{1}{32}\sqrt{\frac{\rho^2 np}{\sigma^2}}\right)$.

\paragraph{Step 2: Error decomposition.} For any $V\in\mathbb{C}_{1}^{n\times n}$ such that $\ell(V,z^*)\leq \gamma$, we can define $\wt{V}\in\mathbb{C}^{n\times n}$ such that  $$\wt{V}_j=\frac{\sum_{k\in[n]\backslash\{j\}}A_{jk}\bar{Y}_{jk}V_k}{\sum_{k\in[n]\backslash\{j\}}A_{jk}},$$ for each $j\in[n]$. Denote $\wh{V}=f(V)$ then $\wh{V}_j=\wt{V}_j/\|\wt{V}_j\|$ for each coordinate such that $\wt{V}_j\neq 0$.

The condition $\ell(V,z^*)\leq \gamma$ implies there exists some $b\in\mathbb{C}^n$ such that $\|b\|=1$ and $n\ell(V,z^*)=\fnorm{V-bz^{*\H}}^2\leq \gamma n$. By direct calculation, we can write
\begin{eqnarray*}
z_j^*\wt{V}_j &=& b + \frac{\sum_{k\in[n]\backslash\{j\}}A_{jk}z_k^*(V_k-\bar{z}_k^*b)}{\sum_{k\in[n]\backslash\{j\}}A_{jk}} + \frac{\sigma z_j^*b\sum_{k\in[n]\backslash\{j\}}A_{jk}\bar{W}_{jk}\bar{z}_k^*}{\sum_{k\in[n]\backslash\{j\}}A_{jk}} \\
&& + \frac{\sigma z_j^*\sum_{k\in[n]\backslash\{j\}}A_{jk}\bar{W}_{jk}(V_k-\bar{z}_k^*b)}{\sum_{k\in[n]\backslash\{j\}}A_{jk}} \\
&=& b + \frac{1}{n-1}\sum_{k=1}^nz_k^*(V_k-\bar{z}_k^*b) - \frac{1}{n-1}z_j^*(V_j-\bar{z}_j^*b) \\
&& + \left(\frac{\sum_{k\in[n]\backslash\{j\}}A_{jk}z_k^*(V_k-\bar{z}_k^*b)}{\sum_{k\in[n]\backslash\{j\}}A_{jk}}-\frac{1}{n-1}\sum_{k=1}^nz_k^*(V_k-\bar{z}_k^*b)\right) \\
&& + \frac{\sigma z_j^*b\sum_{k\in[n]\backslash\{j\}}A_{jk}\bar{W}_{jk}\bar{z}_k^*}{\sum_{k\in[n]\backslash\{j\}}A_{jk}} + \frac{\sigma z_j^*\sum_{k\in[n]\backslash\{j\}}A_{jk}\bar{W}_{jk}(V_k-\bar{z}_k^*b)}{\sum_{k\in[n]\backslash\{j\}}A_{jk}}.
\end{eqnarray*}
Now we define $a_0=b + \frac{1}{n-1}\sum_{k=1}^nz_k^*(V_k-\bar{z}_k^*b)$ and $a=a_0/\|a_0\|$, and we have
\begin{eqnarray}
\label{eq:tetris99} z_j^*a^{\H}\wt{V}_j &=& \|a_0\| - \frac{1}{n-1}z_j^*a^{\H}(V_j-\bar{z}_j^*b) + a^{\H}F_j + a^{\H}bG_j + a^{\H}H_j, \\
\label{eq:tetris100} \|(I_n-aa^{\H})\wt{V}_j\| &\leq& \frac{1}{n-1}\|V_j-\bar{z}_j^*b\| + \|F_j\| + \|(I_n-aa^{\H})b\||G_j| + \|H_j\|,
\end{eqnarray}
where
\begin{eqnarray*}
F_j &=& \frac{\sum_{k\in[n]\backslash\{j\}}A_{jk}z_k^*(V_k-\bar{z}_k^*b)}{\sum_{k\in[n]\backslash\{j\}}A_{jk}}-\frac{1}{n-1}\sum_{k=1}^nz_k^*(V_k-\bar{z}_k^*b), \\
G_j &=& \frac{\sigma z_j^*\sum_{k\in[n]\backslash\{j\}}A_{jk}\bar{W}_{jk}\bar{z}_k^*}{\sum_{k\in[n]\backslash\{j\}}A_{jk}}, \\
H_j &=& \frac{\sigma z_j^*\sum_{k\in[n]\backslash\{j\}}A_{jk}\bar{W}_{jk}(V_k-\bar{z}_k^*b)}{\sum_{k\in[n]\backslash\{j\}}A_{jk}}.
\end{eqnarray*}
By Lemma \ref{lem:middle}, we have the bound
\begin{equation}
\|\wh{V}_j-\bar{z}_j^*a\|^2 \leq \frac{\|(I_n-aa^{\H})\wt{V}_j\|^2 + |\im(z_j^*a^{\H}\wt{V}_j)|^2}{|\re(z_j^*a^{\H}\wt{V}_j)|^2}, \label{eq:by-anderson}
\end{equation}
whenever $\re(z_j^*a^{\H}\wt{V}_j)>0$ holds. Since
$$\|a_0-b\|=\left\|\frac{1}{n-1}\sum_{k=1}^nz_k^*(V_k-\bar{z}_k^*b)\right\|\leq \frac{1}{n-1}\sqrt{n}\fnorm{V-bz^{*\H}}\leq \frac{n}{n-1}\sqrt{\gamma}\leq 2\sqrt{\gamma},$$
we have $\|a-b\|\leq 2\|a_0-b\|\leq 4\sqrt{\gamma}$. Therefore,
\begin{eqnarray}
\label{eq:ab-1} \|a_0\| &\geq& \|b\|-\|a_0-b\| \geq 1-2\sqrt{\gamma}, \\
\label{eq:ab-2} |a^{\H}b-1| &=& |(a-b)^{\H}b| \leq \|a-b\| \leq 4\sqrt{\gamma}, \\
\label{eq:ab-3} \|(I_n-aa^{\H})b\| &\leq& \|a-b\| + |a^{\H}b-1| \leq 8\sqrt{\gamma}.
\end{eqnarray}
We also have
\begin{equation}
\left|\frac{1}{n-1}z_j^*a^{\H}(V_j-\bar{z}_j^*b)\right| \leq \frac{1}{n-1}\|V_j-\bar{z}_j^*b\| \leq \frac{\sqrt{\gamma n}}{n-1}. \label{eq:small-term}
\end{equation}
Therefore, as long as $\|F_j\|\vee|G_j|\vee\|H_j\|\leq \rho$, we have
\begin{equation}
\re\left(z_j^*a^{\H}\wt{V}_j \right)\geq 1-2\sqrt{\gamma} - \frac{\sqrt{\gamma n}}{n-1} -3\rho\geq 1-3(\sqrt{\gamma}+\rho)>0, \label{eq:switch}
\end{equation}
where we have used (\ref{eq:tetris99}), (\ref{eq:ab-1}), and (\ref{eq:small-term}), and the last inequality is due to the assumption that $\gamma <1/16$ and $\rho$ is sufficiently small. Hence, the event $\{\wt{V}_j=0\}$ is included in the event $\{\|F_j\|\vee|G_j|\vee\|H_j\|> \rho\}$.

By (\ref{eq:by-anderson}), we obtain the bound
\begin{eqnarray*}
\|\wh{V}_j-\bar{z}_j^*a\|^2 &\leq& \frac{\|(I_n-aa^{\H})\wt{V}_j\|^2 + |\im(z_j^*a^{\H}\wt{V}_j)|^2}{|\re(z_j^*a^{\H}\wt{V}_j)|^2}\mathbb{I}\{\|F_j\|\vee|G_j|\vee\|H_j\|\leq \rho\} \\
&& + 4\mathbb{I}\{\|F_j\|\vee|G_j|\vee\|H_j\|> \rho\} \\
&\leq& \frac{\left(\frac{1}{n-1}\|V_j-\bar{z}_j^*b\| + \|F_j\| + 8\sqrt{\gamma}|G_j| + \|H_j\|\right)^2}{(1-3(\sqrt{\gamma}+\rho))^2} \\
&& + \frac{\left(\frac{1}{n-1}\|V_j-\bar{z}_j^*b\| + \|F_j\| + |\im(a^{\H}bG_j)| + \|H_j\|\right)^2}{(1-3(\sqrt{\gamma}+\rho))^2} \\
&& + 4\mathbb{I}\{\|F_j\|>\rho\} + 4\mathbb{I}\{|G_j|>\rho\} + 4\mathbb{I}\{\|H_j\|>\rho\} \\
&\leq& \frac{(1+\eta)|\im(a^{\H}bG_j)|^2 +256\gamma|G_j|^2}{(1-3(\sqrt{\gamma}+\rho))^2} \\
&& + \frac{(7+4\eta^{-1})\left(\frac{1}{(n-1)^2}\|V_j-\bar{z}_j^*b\|^2+\|F_j\|^2+\|H_j\|^2\right)}{(1-3(\sqrt{\gamma}+\rho))^2} \\
&& + 4\mathbb{I}\{\|F_j\|>\rho\} + 4\mathbb{I}\{|G_j|>\rho\} + 4\mathbb{I}\{\|H_j\|>\rho\},
\end{eqnarray*}
for some $\eta$ to be specified later. The last inequality above is due to Jensen's inequality.

\paragraph{Step 3: Analysis of each error term.} Next, we will analyze the error terms $F_j$, $H_j$ and $G_j$ separately. By triangle inequality, (\ref{eq:hollow1}) and (\ref{eq:hollow2}), we have
\begin{eqnarray*}
\|F_j\| &\leq& \left\|\frac{\sum_{k\in[n]\backslash\{j\}}(A_{jk}-p)z_k^*(V_k-\bar{z}_k^*b)}{\sum_{k\in[n]\backslash\{j\}}A_{jk}}\right\| \\
&& + \left\|p\sum_{k\in[n]\backslash\{j\}}z_k^*(V_k-\bar{z}_k^*b)\right\|\left|\frac{1}{\sum_{k\in[n]\backslash\{j\}}A_{jk}}-\frac{1}{(n-1)p}\right| \\
&\leq& \frac{2}{np}\left\|\sum_{k\in[n]\backslash\{j\}}(A_{jk}-p)z_k^*(V_k-\bar{z}_k^*b)\right\| + p\sqrt{n}\fnorm{V-bz^{*\H}}\frac{2\left|\sum_{k\in[n]\backslash\{j\}}(A_{jk}-p)\right|}{n^2p^2} \\
&\leq& \frac{2}{np}\left\|\sum_{k\in[n]\backslash\{j\}}(A_{jk}-p)z_k^*(V_k-\bar{z}_k^*b)\right\| + C_1\frac{\sqrt{p\log n}}{np}\fnorm{V-bz^{*\H}}.
\end{eqnarray*}
Using (\ref{eq:hollow3}), we have
\begin{eqnarray}
\nonumber \sum_{j=1}^n\|F_j\|^2 &\leq& \frac{8}{n^2p^2}\sum_{j=1}^n\left\|\sum_{k\in[n]\backslash\{j\}}(A_{jk}-p)z_k^*(V_k-\bar{z}_k^*b)\right\| + 2C_1^2\frac{\log n}{np}\fnorm{V-bz^{*\H}}^2 \\
\nonumber &\leq& \frac{8}{n^2p^2}\opnorm{A-\mathbb{E}A}^2\fnorm{V-bz^{*\H}}^2  + 2C_1^2\frac{\log n}{np}\fnorm{V-bz^{*\H}}^2 \\
\label{eq:F-bound} &\leq& C_2\frac{\log n}{np}\fnorm{V-bz^{*\H}}^2.
\end{eqnarray}
The above bound also implies
$$\sum_{j=1}^n\mathbb{I}\{\|F_j\|>\rho\}\leq \rho^{-2}\sum_{j=1}^n\|F_j\|^2\leq \frac{C_2}{\rho^2}\frac{\log n}{np}\fnorm{V-bz^{*\H}}^2.$$
Similarly, we can also bound the error terms that depend on $H_j$. By (\ref{eq:hollow1}) and (\ref{eq:hollow4}), we have
\begin{eqnarray}
\nonumber \sum_{j=1}^n\|H_j\|^2 &\leq& \frac{2\sigma^2}{n^2p^2}\sum_{j=1}^n\left\|\sum_{k\in[n]\backslash\{j\}}A_{jk}\bar{W}_{jk}(V_k-\bar{z}_k^*b)\right\|^2 \\
\nonumber  &=& \frac{2\sigma^2}{n^2p^2}\fnorm{(V-bz^{*\H})(A\circ W)^{\H}}^2 \\
\nonumber &\leq& \frac{2\sigma^2}{n^2p^2}\opnorm{A\circ W}^2\fnorm{V-bz^{*\H}}^2 \\
\label{eq:H-bound} &\leq& C_3\frac{\sigma^2}{np}\fnorm{V-bz^{*\H}}^2,
\end{eqnarray}
and thus
$$\sum_{j=1}^n\mathbb{I}\{\|H_j\|>\rho\}\leq \rho^{-2}\sum_{j=1}^n\|H_j\|^2\leq \frac{C_3}{\rho^2}\frac{\sigma^2}{np}\fnorm{V-bz^{*\H}}^2.$$
For the contribution of $G_j$, we use (\ref{eq:hollow1}) and (\ref{eq:hollow5}), and have
\begin{eqnarray}
\nonumber \sum_{j=1}^n\mathbb{I}\{|G_j|>\rho\} &\leq& \sum_{j=1}^n\mathbb{I}\left\{\frac{2\sigma}{np}\left|\sum_{k\in[n]\backslash\{j\}}A_{jk}\bar{W}_{jk}\bar{z}_k^*\right|>\rho\right\} \\
\label{eq:G-exp} &\leq& \frac{4\sigma^2}{\rho^2p} \exp\left(-\frac{1}{16}\sqrt{\frac{\rho^2 np}{\sigma^2}}\right).
\end{eqnarray}
Next, we study the main error term $|\im(a^{\H}bG_j)|^2$. By (\ref{eq:hollow1}), we have
\begin{eqnarray*}
\sum_{j=1}^n|\im(a^{\H}bG_j)|^2 &\leq& \left(1+C_4\sqrt{\frac{\log n}{np}}\right)^2\frac{\sigma^2}{n^2p^2}\sum_{j=1}^n\left|\sum_{k\in[n]\backslash\{j\}}A_{jk}\im(\bar{W}_{jk}z_j^*\bar{z}_k^*a^{\H}b)\right|^2 \\
&\leq& (1+\eta)\left(1+C_4\sqrt{\frac{\log n}{np}}\right)^2\frac{\sigma^2}{n^2p^2}\sum_{j=1}^n\left|\sum_{k\in[n]\backslash\{j\}}A_{jk}\im(\bar{W}_{jk}z_j^*\bar{z}_k^*)\right|^2 \\
&& + (1+\eta^{-1})\left(1+C_4\sqrt{\frac{\log n}{np}}\right)^2\frac{\sigma^2}{n^2p^2}\sum_{j=1}^n\left|\sum_{k\in[n]\backslash\{j\}}A_{jk}\re(\bar{W}_{jk}z_j^*\bar{z}_k^*)\right|^2|\im(a^{\H}b)|^2.
\end{eqnarray*}
By (\ref{eq:ab-2}), we have
$$|\im(a^{\H}b)|=|\im(a^{\H}b-1)|\leq |a^{\H}b-1|\leq 4\sqrt{\gamma}.$$
Together with (\ref{eq:hollow6}) and (\ref{eq:hollow7}), we have
\begin{equation}
\sum_{j=1}^n|\im(a^{\H}bG_j)|^2 \leq \left(1+C_5\left(\eta+\eta^{-1}\gamma+\sqrt{\frac{\log n}{np}}\right)\right)\frac{\sigma^2}{2p}.\label{eq:tenergy}
\end{equation}
We also have
\begin{equation}
\sum_{j=1}^n|G_j|^2 \leq \frac{2\sigma^2}{n^2p^2}\sum_{j=1}^n\left|\sum_{k\in[n]\backslash\{j\}}A_{jk}\bar{W}_{jk}z_j^*\bar{z}_k^*\right|^2 \leq C_6\frac{\sigma^2}{p}, \label{eq:G-bound}
\end{equation}
by (\ref{eq:hollow6}) and (\ref{eq:hollow7}).

\paragraph{Step 4: Combining the bounds.} Plugging all the individual error bounds obtained in Step 3 into the error decomposition in Step 2, we obtain
\begin{eqnarray*}
n\ell(\wh{V},z^*) &\leq& \sum_{j=1}^n\|\wh{V}_j-\bar{z}_j^*a\|^2 \\
&\leq& \left(1+C_7\left(\rho+\eta+\sqrt{\gamma}+\eta^{-1}\gamma+\sqrt{\frac{\log n}{np}}\right)\right)\frac{\sigma^2}{2p} \\
&& + \frac{16\sigma^2}{\rho^2 p}\exp\left(-\frac{1}{16}\sqrt{\frac{\rho^2np}{\sigma^2}}\right)  + C_7\left(\eta^{-1}+\rho^{-2}\right)\frac{\log n+\sigma^2}{np}n\ell(V,z^*).
\end{eqnarray*}
We set
$$\eta=\sqrt{\gamma+\frac{\log n+\sigma^2}{np}}\quad\text{and}\quad \rho^2=\sqrt{32}\sqrt{\frac{\log n+\sigma^2}{np}}.$$
Then, since $\frac{\rho^2np}{\sigma^2}$ is sufficiently large, we have
$$\frac{16\sigma^2}{\rho^2 p}\exp\left(-\frac{1}{16}\sqrt{\frac{\rho^2np}{\sigma^2}}\right)\leq \frac{\sigma^2}{\rho^2 p}\left(\frac{\sigma^2}{\rho^2 np}\right)^2\leq\frac{\sigma^2}{p}\sqrt{\frac{\sigma^2}{np}}.$$
Therefore, we have
$$\ell(\wh{V},z^*)\leq \left(1+C_8\left(\gamma^2+\frac{\log n+\sigma^2}{np}\right)^{1/4}\right)\frac{\sigma^2}{2np} + C_8\sqrt{\frac{\log n+\sigma^2}{np}}\ell(V,z^*).$$
Since the above inequality is derived from the conditions (\ref{eq:hollow1})-(\ref{eq:hollow5}) and $\ell(V,z^*)\leq \gamma$, it holds uniformly over all $V\in\mathbb{C}_{1}^{n\times n}$ such that $\ell(V,z^*)\leq \gamma$ with probability at least $1-n^{-9}-\exp\left(-\left(\frac{np}{\sigma^2}\right)^{1/4}\right)$. The proof is complete.

\subsection{Proof of Lemma \ref{lem:critical-z2}}\label{sec:pf-cri-lem-z2}

Similar to the proof of Lemma \ref{lem:critical}, we organize the proof of Lemma \ref{lem:critical-z2} into four steps.

\paragraph{Step 1: Some high-probability events.} We already know that (\ref{eq:hollow1}), (\ref{eq:hollow2}) and (\ref{eq:hollow3}) hold with probability at least $1-n^{-9}$. We also have
\begin{equation}
\opnorm{A\circ W} \leq C\sqrt{np}, \label{eq:hollow4-z2}
\end{equation}
with probability at least $1-n^{-9}$ by Lemma \ref{lem:bandeira}. Note that the matrix $W$ in (\ref{eq:hollow4-z2}) is real-valued, compared with the complex version of the bound (\ref{eq:hollow4}). Another high-probability event we need is for the random variable $\sum_{j=1}^n\left|\sum_{k\in[n]\backslash\{j\}}A_{jk}W_{jk}z_j^*z_k^*\right|^2$. By Lemma \ref{lem:talagrand} and with a similar analysis that leads to (\ref{eq:hollow7}), 
we can conclude that with probability at least $1-n^{-9}$,
\begin{equation}
\sum_{j=1}^n\left|\sum_{k\in[n]\backslash\{j\}}A_{jk}W_{jk}z_j^*z_k^*\right|^2 \leq  n^2p\left(1+C\sqrt{\frac{\log n}{n}}\right). \label{eq:hollow5-z2}
\end{equation}
In the end, we conclude that the events (\ref{eq:hollow1}), (\ref{eq:hollow2}), (\ref{eq:hollow3}), (\ref{eq:hollow4-z2}) and (\ref{eq:hollow5-z2}) hold simultaneously with probability at least $1-2n^{-9}$.

\paragraph{Step 2: Error decomposition.} For any $V\in\mathbb{R}_{1}^{n\times n}$ such that $\ell(V,z^*)\leq \gamma$, we can write $\wh{V}=f(V)$ with each column $\wh{V}_j=\wt{V}_j/\|\wt{V}_j\|$, where
$$\wt{V}_j=\frac{\sum_{k\in[n]\backslash\{j\}}A_{jk}{Y}_{jk}V_k}{\sum_{k\in[n]\backslash\{j\}}A_{jk}}.$$
The condition $\ell(V,z^*)\leq \gamma$ implies there exists some $b\in\mathbb{R}^n$ such that $\|b\|=1$ and $n\ell(V,z^*)=\fnorm{V-bz^{*\T}}^2\leq \gamma n$. By direct calculation, we can write
\begin{eqnarray*}
z_j^*\wt{V}_j &=& b + \frac{1}{n-1}\sum_{k=1}^nz_k^*(V_k-{z}_k^*b) - \frac{1}{n-1}z_j^*(V_j-{z}_j^*b) \\
&& + \left(\frac{\sum_{k\in[n]\backslash\{j\}}A_{jk}z_k^*(V_k-{z}_k^*b)}{\sum_{k\in[n]\backslash\{j\}}A_{jk}}-\frac{1}{n-1}\sum_{k=1}^nz_k^*(V_k-{z}_k^*b)\right) \\
&& + \frac{\sigma z_j^*b\sum_{k\in[n]\backslash\{j\}}A_{jk}{W}_{jk}{z}_k^*}{\sum_{k\in[n]\backslash\{j\}}A_{jk}} + \frac{\sigma z_j^*\sum_{k\in[n]\backslash\{j\}}A_{jk}{W}_{jk}(V_k-{z}_k^*b)}{\sum_{k\in[n]\backslash\{j\}}A_{jk}}.
\end{eqnarray*}
Now we define $a_0=b + \frac{1}{n-1}\sum_{k=1}^nz_k^*(V_k-{z}_k^*b)$ and $a=a_0/\|a_0\|$, and we have
\begin{eqnarray*}
z_j^*a^{\T}\wt{V}_j &=& \|a_0\| - \frac{1}{n-1}z_j^*a^{\T}(V_j-{z}_j^*b) + a^{\T}F_j + a^{\T}bG_j + a^{\T}H_j, \\
\|(I_n-aa^{\T})\wt{V}_j\| &\leq& \frac{1}{n-1}\|V_j-{z}_j^*b\| + \|F_j\| + \|(I_n-aa^{\T})b\||G_j| + \|H_j\|,
\end{eqnarray*}
where
\begin{eqnarray*}
F_j &=& \frac{\sum_{k\in[n]\backslash\{j\}}A_{jk}z_k^*(V_k-{z}_k^*b)}{\sum_{k\in[n]\backslash\{j\}}A_{jk}}-\frac{1}{n-1}\sum_{k=1}^nz_k^*(V_k-{z}_k^*b), \\
G_j &=& \frac{\sigma z_j^*\sum_{k\in[n]\backslash\{j\}}A_{jk}{W}_{jk}{z}_k^*}{\sum_{k\in[n]\backslash\{j\}}A_{jk}}, \\
H_j &=& \frac{\sigma z_j^*\sum_{k\in[n]\backslash\{j\}}A_{jk}{W}_{jk}(V_k-{z}_k^*b)}{\sum_{k\in[n]\backslash\{j\}}A_{jk}}.
\end{eqnarray*}
By Lemma \ref{lem:middle}, we have the bound
\begin{equation}
\|\wh{V}_j-{z}_j^*a\|^2 \leq \frac{\|(I_n-aa^{\T})\wt{V}_j\|^2}{|z_j^*a^{\T}\wt{V}_j|^2}, \label{eq:by-anderson-z2}
\end{equation}
whenever $z_j^*a^{\T}\wt{V}_j>0$ holds. Since
$$\|a_0-b\|=\left\|\frac{1}{n-1}\sum_{k=1}^nz_k^*(V_k-{z}_k^*b)\right\|\leq \frac{1}{n-1}\sqrt{n}\fnorm{V-bz^{*\T}}\leq \frac{2}{\sqrt{n}}\fnorm{V-bz^{*\T}},$$
we have $\|a-b\|\leq 2\|a_0-b\|\leq \frac{4}{\sqrt{n}}\fnorm{V-bz^{*\T}}$. Therefore,
\begin{eqnarray}
\label{eq:ab-1-z2} \|a_0\| &\geq& \|b\|-\|a_0-b\| \geq 1-\frac{2}{\sqrt{n}}\fnorm{V-bz^{*\T}}, \\
\label{eq:ab-2-z2} |a^{\T}b-1| &=& |(a-b)^{\T}b| \leq \|a-b\| \leq \frac{4}{\sqrt{n}}\fnorm{V-bz^{*\T}}, \\
\label{eq:ab-3-z2} \|(I_n-aa^{\T})b\| &\leq& \|a-b\| + |a^{\T}b-1| \leq \frac{8}{\sqrt{n}}\fnorm{V-bz^{*\T}}.
\end{eqnarray}
We also have
\begin{equation}
\left|\frac{1}{n-1}z_j^*a^{\T}(V_j-{z}_j^*b)\right| \leq \frac{1}{n-1}\|V_j-{z}_j^*b\| \leq \frac{1}{n-1}\fnorm{V-bz^{*\T}}. \label{eq:small-term-z2}
\end{equation}
Therefore, as long as $\|F_j\|\vee\|H_j\|\leq \rho$ and $|G_j|\leq 1-4\rho$, we have
\begin{eqnarray*}
z_j^*a^{\T}\wt{V}_j &\geq& 1-\frac{2}{\sqrt{n}}\fnorm{V-bz^{*\T}}-\frac{1}{n-1}\fnorm{V-bz^{*\T}}-\|F_j\|-|G_j|-\|H_j\| \\
&\geq& 1-3\sqrt{\gamma}-2\rho-(1-4\rho) \\
&\geq& \rho,
\end{eqnarray*}
where we have used (\ref{eq:ab-1-z2}) and (\ref{eq:small-term-z2}), and we set $\rho$ to satisfy $\rho\geq 3\sqrt{\gamma}$. The specific choice of $\rho$ will be given later. Hence, the event $\{\wt{V}_j=0\}$ is included in the event $\{\|F_j\|\vee\|H_j\|> \rho \text{ or } |G_j|> 1-4\rho\}$.

By (\ref{eq:by-anderson-z2}), we obtain the bound
\begin{eqnarray*}
\|\wh{V}_j-{z}_j^*a\|^2 &\leq& \frac{\|(I_n-aa^{\T})\wt{V}_j\|^2}{|z_j^*a^{\T}\wt{V}_j|^2}\mathbb{I}\{\|F_j\|\vee\|H_j\|\leq \rho, |G_j|\leq 1-4\rho\} \\
&& + 4\mathbb{I}\{\|F_j\|\vee\|H_j\|> \rho \text{ or } |G_j|> 1-4\rho\} \\
&\leq& \frac{1}{\rho^2}\left(\frac{1}{n-1}\|V_j-{z}_j^*b\| + \|F_j\| + \|(I_n-aa^{\T})b\||G_j| + \|H_j\|\right)^2 \\
&& + 4\mathbb{I}\{\|F_j\|>\rho\} + 4\mathbb{I}\{|G_j|>1-4\rho\} + 4\mathbb{I}\{\|H_j\|>\rho\} \\
&\leq& \frac{4\|V_j-{z}_j^*b\|^2}{\rho^2(n-1)^2} + \frac{4\|F_j\|^2}{\rho^2} + \frac{4\|(I_n-aa^{\T})b\|^2|G_j|^2}{\rho^2} + \frac{4\|H_j\|^2}{\rho^2} \\
&& + \frac{4\|F_j\|^2}{\rho^2} + \frac{4\|H_j\|^2}{\rho^2} + 4\mathbb{I}\{|G_j|>1-4\rho\} \\
&\leq&  \frac{4\|V_j-{z}_j^*b\|^2}{\rho^2(n-1)^2} + \frac{8\|F_j\|^2}{\rho^2} + \frac{256\fnorm{V-bz^{*\T}}^2|G_j|^2}{n\rho^2} + \frac{8\|H_j\|^2}{\rho^2} \\
&& + 4\mathbb{I}\{|G_j|>1-4\rho\}.
\end{eqnarray*}
We have used (\ref{eq:ab-3-z2}), Jensen's inequality and Markov's inequality in the above derivation.

\paragraph{Step 3: Analysis of each error term.} Next, we will analyze the error terms $F_j$, $H_j$ and $G_j$ separately. Following the same analysis that leads to (\ref{eq:F-bound}), (\ref{eq:H-bound}) and (\ref{eq:G-bound}), we have
\begin{eqnarray*}
\sum_{j=1}^n\|F_j\|^2 &\leq& C_1\frac{\log n}{np}\fnorm{V-bz^{*\H}}^2, \\
\sum_{j=1}^n\|H_j\|^2 &\leq& C_2\frac{\sigma^2}{np}\fnorm{V-bz^{*\H}}^2, \\
\sum_{j=1}^n|G_j|^2 &\leq& C_3\frac{\sigma^2}{p}.
\end{eqnarray*}
Note that the above three bounds are based on the events (\ref{eq:hollow1}), (\ref{eq:hollow2}), (\ref{eq:hollow3}), (\ref{eq:hollow4-z2}) and (\ref{eq:hollow5-z2}).

\paragraph{Step 4: Combining the bounds.} Plugging all the individual error bounds obtained in Step 3 into the error decomposition in Step 2, we obtain
\begin{eqnarray*}
n\ell(\wh{V},z^*) &\leq& \sum_{j=1}^n\|\wh{V}_j-{z}_j^*a\|^2 \\
&\leq& \left(\frac{4}{\rho^2(n-1)^2}+\frac{8C_1\log n+(8C_2+256C_3)\sigma^2}{\rho^2np}\right)n\ell(V,z^*) + 4\sum_{j=1}^n\mathbb{I}\{|G_j|>1-4\rho\}.
\end{eqnarray*}
Set
$$\rho^2=\left(C_4\frac{\log n+\sigma^2}{np}\right) \vee  \br{3\gamma},$$
for some sufficiently large constant $C_4$ such that $\frac{4}{\rho^2(n-1)^2}+\frac{8C_1\log n+(8C_2+256C_3)\sigma^2}{\rho^2np}\leq \frac{1}{2}$. Then, we have
\begin{eqnarray*}
\ell(\wh{V},z^*) &\leq& \frac{1}{2}\ell(V,z^*)+\frac{4}{n}\sum_{j=1}^n\mathbb{I}\{|G_j|>1-4\rho\} \\
&\leq& \frac{1}{2}\ell(V,z^*)+\frac{4}{n}\sum_{j=1}^n\mathbb{I}\left\{\frac{\sigma}{(n-1)p}\left|\sum_{k\in[n]\backslash\{j\}}z_k^*A_{jk}W_{jk}\right|>\left(1-C_5\sqrt{\frac{\log n+\sigma^2}{np}} -\sqrt{3\gamma} \right)\right\},
\end{eqnarray*}
where the last inequality is by (\ref{eq:hollow1}) and (\ref{eq:hollow2}). Note that $\frac{\log n +\sigma^2}{np}$ is sufficiently small and $\gamma <1/16$, we have $\delta <1$ with $\delta = C_5\sqrt{\frac{\log n+\sigma^2}{np}} +\sqrt{3\gamma}$.
Since the above about is derived from the conditions (\ref{eq:hollow1}), (\ref{eq:hollow2}), (\ref{eq:hollow3}), (\ref{eq:hollow4-z2}) and (\ref{eq:hollow5-z2}) and $\ell(V,z^*)\leq \gamma$, it holds uniformly over all $V\in\mathbb{R}_{1}^{n\times n}$ such that $\ell(V,z^*)\leq \gamma$ with probability at least $1-2n^{-9}$. The proof is complete.

\subsection{Proofs of Theorem \ref{thm:main1}, Theorem \ref{thm:round-phase}, and Theorem \ref{thm:main-z2}}\label{sec:pf-main}

\begin{proof}[Proof of Theorem \ref{thm:main1}]
We obtain (\ref{eq:apply-l-2-SDP}) as a consequence of Lemma \ref{lem:critical} and Lemma \ref{lem:SDP-crude}, which immediately implies the first conclusion. The second conclusion is a consequence of Lemma \ref{lem:loss-relation}.
\end{proof}

\begin{proof}[Proof of Theorem \ref{thm:round-phase}]
By Theorem \ref{thm:main1}, we have $\fnorm{\wh{V}-bz^{*\H}}^2\leq \frac{\sigma^2}{p}$ with high probability for some $b\in\mathbb{C}^n$ such that $\|b\|=1$. Since $\wh{V}=f(\wh{V})$, we can follow the same analysis in the proof of Lemma \ref{lem:critical} and obtain the bound
\begin{equation}
\fnorm{\wh{V}-az^{*\H}}^2 \leq (1+\delta)\frac{\sigma^2}{2p},\label{eq:bound-forgot}
\end{equation}
with high probability, where $\delta=C\left(\frac{\log n+\sigma^2}{np}\right)^{1/4}$ and $a=a_0/\|a_0\|$ with $a_0=b + \frac{1}{n-1}\sum_{k=1}^nz_k^*(\wh{V}_k-\bar{z}_k^*b)$. Let $\wt{z}=\wh{V}^{\H}\wh{a}$ where $\wh{a}$ is the leading left singular vector of $\wh{V}$. By the definition of $\wh{z}$, we can write $\wh{z}_j=\wt{z}_j/|\wt{z}_j|$ for all $j\in[n]$ with non-zero $\wt{z}_j$.

 By (\ref{eq:bound-forgot}) and Wedin's sin-theta theorem \citep{wedin1972perturbation}, we have
\begin{equation}
\|\wh{a}-ha\|^2 \leq \frac{\fnorm{\wh{V}-az^{*\H}}^2 }{n}\leq   \frac{\sigma^2}{np},\label{eq:apply-wedin}
\end{equation}
form some $h\in\mathbb{C}_1$. Define $d_0=a^{\H}\wh{a}$ and $d=d_0/|d_0|$. With $\wt{z}_j=\wh{V}_j^{\H}\wh{a}$, we have
\begin{equation}
\wt{z}_j\bar{z}_j^*\bar{d}=|d_0|+h\bar{d}\bar{z}_j^*(\wh{V}_j-\bar{z}_j^*a)^{\H}a+\bar{d}(\wh{V}_j-\bar{z}_j^*a)^{\H}(\wh{a}-ha)\bar{z}_j^*. \label{eq:round-basic}
\end{equation}
By Lemma \ref{lem:middle}, we have the bound
\begin{equation}
|\wh{z}_j-dz_j^*|^2 = \left|\frac{\wt{z}_j\bar{z}_j^*\bar{d}}{|\wt{z}_j\bar{z}_j^*\bar{d}|}-1\right|^2 \leq \frac{|\im(\wt{z}_j\bar{z}_j^*\bar{d})|^2}{|\re(\wt{z}_j\bar{z}_j^*\bar{d})|^2}, \label{eq:first-ratio-bound}
\end{equation}
as long as $\re(\wt{z}_j\bar{z}_j^*\bar{d})>0$. By (\ref{eq:apply-wedin}), we have $|d_0|\geq\re(h\wh{a}^{\H}a)\geq 1-\frac{\sigma^2}{2np}$ and $|\bar{d}(\wh{V}_j-\bar{z}_j^*a)^{\H}(\wh{a}-ha)\bar{z}_j^*|\leq \sqrt{\frac{\sigma^2}{np}}\|\wh{V}_j-\bar{z}_j^*a\|$. Moreover,
$$|h\bar{d}\bar{z}_j^*(\wh{V}_j-\bar{z}_j^*a)^{\H}a|\leq |(\wh{V}_j-\bar{z}_j^*a)^{\H}a|=|\wh{V}_j^{\H}a\bar{z}_j^*-1|=|a^{\H}\wh{V}_jz_j^*-1|.$$
Therefore,
\begin{equation}
\re(\wt{z}_j\bar{z}_j^*\bar{d})\geq 1-\frac{\sigma^2}{2np}-|a^{\H}\wh{V}_jz_j^*-1|-\sqrt{\frac{\sigma^2}{np}}\|\wh{V}_j-\bar{z}_j^*a\|. \label{eq:den-round}
\end{equation}

In the following, we are going to establish a lower bound for (\ref{eq:den-round}) using some similar analysis as in the proof of Lemma \ref{lem:critical}. 
Since $\wh{V}=f(\wh{V})$, we can write $\wh{V}_j=\wt{V}_j/\|\wt{V}_j\|$ for all non-zero $\wt{V}_j$, where
$$\wt{V}_j=\frac{\sum_{k\in[n]\backslash\{j\}}A_{jk}{Y}_{jk}\wh{V}_k}{\sum_{k\in[n]\backslash\{j\}}A_{jk}}.$$
Similar to the decomposition (\ref{eq:tetris99}), we can write
$$a^{\H}\wt{V}_jz_j^* = \|a_0\| - \frac{1}{n-1}z_j^*a^{\H}(\wh{V}_j-\bar{z}_j^*b) + a^{\H}F_j + a^{\H}bG_j + a^{\H}H_j,$$
where
\begin{eqnarray*}
F_j &=& \frac{\sum_{k\in[n]\backslash\{j\}}A_{jk}z_k^*(\wh{V}_k-\bar{z}_k^*b)}{\sum_{k\in[n]\backslash\{j\}}A_{jk}}-\frac{1}{n-1}\sum_{k=1}^nz_k^*(\wh{V}_k-\bar{z}_k^*b), \\
G_j &=& \frac{\sigma z_j^*\sum_{k\in[n]\backslash\{j\}}A_{jk}\bar{W}_{jk}\bar{z}_k^*}{\sum_{k\in[n]\backslash\{j\}}A_{jk}}, \\
H_j &=& \frac{\sigma z_j^*\sum_{k\in[n]\backslash\{j\}}A_{jk}\bar{W}_{jk}(\wh{V}_k-\bar{z}_k^*b)}{\sum_{k\in[n]\backslash\{j\}}A_{jk}}.
\end{eqnarray*}
By the same argument that leads to (\ref{eq:switch}) with $\gamma=\frac{\sigma^2}{np}$, for any $\rho>0$, we know that as long as $\|F_j\|\vee|G_j|\vee\|H_j\|\leq \rho$, we have
\begin{equation}
|\re(a^{\H}\wt{V}_jz_j^*)-1|\leq 3\rho+3\sqrt{\frac{\sigma^2}{np}}. \label{eq:lsa}
\end{equation}
Moreover,
\begin{eqnarray}
\nonumber |\im(a^{\H}\wt{V}_jz_j^*)| &\leq& \frac{1}{n-1}\|\wh{V}_j-\bar{z}_j^*b\|+\|F_j\|+|\im(a^{\H}bG_j)|^2+\|H_j\| \\
\nonumber&\leq&\frac{1}{n-1}\sqrt{\frac{\sigma^2}{p}}+\|F_j\|+|\im(a^{\H}bG_j)|^2+\|H_j\| \\
\label{eq:homepod}  &\leq& \frac{1}{n-1}\sqrt{\frac{\sigma^2}{p}}+3\rho.
\end{eqnarray}
By a similar bound to (\ref{eq:tetris100}), we also have
$$\|(I_n-aa^{\H})\wt{V}_j\|\leq \frac{1}{n-1}\|\wh{V}_j-\bar{z}_j^*b\|+\|F_j\|+|G_j|+\|H_j\|\leq \frac{1}{n-1}\sqrt{\frac{\sigma^2}{p}}+3\rho.$$
With the decomposition $\|\wt{V}_j\|^2=\|(I_n-aa^{\H})\wt{V}_j\|^2+|\im(a^{\H}\wt{V}_jz_j^*)|^2+|\re(a^{\H}\wt{V}_jz_j^*)|^2$, we have
\begin{equation}
\left|\|\wt{V}_j\|^2-1\right|\leq \|(I_n-aa^{\H})\wt{V}_j\|^2+|\im(a^{\H}\wt{V}_jz_j^*)|^2+\left||\re(a^{\H}\wt{V}_jz_j^*)|^2-1\right|\leq 4\rho+4\sqrt{\frac{\sigma^2}{np}}.\label{eq:airpods}
\end{equation}
Let $\rho$ be a sufficiently small with explicit expression to be given later. Together with the assumption that  $\frac{\sigma^2}{np}$ is also sufficiently small, both (\ref{eq:lsa}) and (\ref{eq:airpods}) can be upper bounded by $1/2$, which implies $\re(a^{\H}\wt{V}_jz_j^*)> 1/2$ and $\wt{V_j}\neq 0$ respectively.
Then $\wh{V}_j = \wt{V}_j/\|\wt{V}_j\|$ leads to the bound
\begin{align*}
|a^{\H}\wh{V}_jz_j^*-1| &\leq \left|\frac{\re(a^{\H}\wt{V}_jz_j^*)-\|\wt{V}_j\|}{\|\wt{V}_j\|}\right|+\frac{|\im(a^{\H}\wt{V}_jz_j^*)|}{\|\wt{V}_j\|} \\
&\leq   \left|\frac{\re(a^{\H}\wt{V}_jz_j^*)-1}{\|\wt{V}_j\|}\right|   +    \left|\frac{1-\|\wt{V}_j\|}{\|\wt{V}_j\|}\right|   +\frac{|\im(a^{\H}\wt{V}_jz_j^*)|}{\|\wt{V}_j\|}  \\
&\leq C_1\left(\rho+\sqrt{\frac{\sigma^2}{np}}\right),
\end{align*}
where we use (\ref{eq:lsa})-(\ref{eq:airpods}).
By Lemma \ref{lem:middle}, we have the bound
\begin{equation}
\|\wh{V}_j-\bar{z}_j^*a\|^2 \leq \frac{\|(I_n-aa^{\H})\wt{V}_j\|^2 + |\im(z_j^*a^{\H}\wt{V}_j)|^2}{|\re(z_j^*a^{\H}\wt{V}_j)|^2}\leq C_2\left(\rho^2 + \frac{\sigma^2}{n^2p}\right).\label{eq:michigan}
\end{equation}
Plugging the above two bounds into (\ref{eq:den-round}), we have
$$\re(\wt{z}_j\bar{z}_j^*\bar{d})\geq 1-C_3\left(\rho+\sqrt{\frac{\sigma^2}{np}}\right),$$
which is positive since $\rho$ and $\frac{\sigma^2}{np}$ are sufficiently small. Therefore, we have $\re(\wt{z}_j\bar{z}_j^*\bar{d})>0$ and the bound (\ref{eq:first-ratio-bound}) holds when $\|F_j\|\vee|G_j|\vee\|H_j\|\leq \rho$. Also this implies the event $\{\wt{z}_j=0\}$ is included in the event $\{\|F_j\|\vee|G_j|\vee\|H_j\| > \rho\}$. As a consequence, we have
$$|\wh{z}_j-dz_j^*|^2\leq \frac{|\im(\wt{z}_j\bar{z}_j^*\bar{d})|^2}{\left(1-C_3\left(\rho+\sqrt{\frac{\sigma^2}{np}}\right)\right)^2}\mathbb{I}\{\|F_j\|\vee|G_j|\vee\|H_j\|\leq\rho\} + 4\mathbb{I}\{\|F_j\|\vee|G_j|\vee\|H_j\|>\rho\}.$$

Now we need to bound $|\im(\wt{z}_j\bar{z}_j^*\bar{d})|$ according to the expansion (\ref{eq:round-basic}). We have
\begin{equation}
|\im(\wt{z}_j\bar{z}_j^*\bar{d})|\leq \left|\im(\bar{z}_j^*(\wh{V}_j-\bar{z}_j^*a)^{\H}a)\right|+|\im(h\bar{d})|\left|\re(\bar{z}_j^*(\wh{V}_j-\bar{z}_j^*a)^{\H}a)\right|+\|\wh{V}_j-\bar{z}_j^*a\|\|\wh{a}-ha\|. \label{eq:im-expansion-round}
\end{equation}
By (\ref{eq:apply-wedin}) and (\ref{eq:michigan}), the third term in the bound (\ref{eq:im-expansion-round}) can be further bounded by $C_2\sqrt{\frac{\sigma^2}{np}}\left(\rho+\sqrt{\frac{\sigma^2}{n^2p}}\right)$. To bound the second term on the right hand side of (\ref{eq:im-expansion-round}), we have $|\im(h\bar{d})|\leq |\im(ha^{\H}\wh{a})|\leq \sqrt{1-|\re(ha^{\H}a)|^2}\leq \sqrt{\frac{\sigma^2}{np}}$ by (\ref{eq:apply-wedin}). Together with (\ref{eq:lsa}), we obtain the bound $3\sqrt{\frac{\sigma^2}{np}}\left(\rho+\sqrt{\frac{\sigma^2}{np}}\right)$. By (\ref{eq:airpods}), we can bound the first term in the bound (\ref{eq:im-expansion-round}) by
$$\left|\im(\bar{z}_j^*(\wh{V}_j-\bar{z}_j^*a)^{\H}a)\right|=\left|\im(\wh{V}_j^{\H}a\bar{z}_j^*)\right| = \left| \im(a^{\H}\wh{V}_jz_j^*) \right|\leq \frac{\left|\im(a^{\H}\wt{V}_jz_j^*)\right|}{1-4\left(\rho+\sqrt{\frac{\sigma^2}{np}}\right)}.$$
Then, we have
\begin{eqnarray*}
|\wh{z}_j-dz_j^*|^2 &\leq& \left(1+C_4\left(\rho+\sqrt{\frac{\sigma^2}{np}}\right)\right)\left|\im(a^{\H}\wt{V}_jz_j^*)\right|^2 + C_5\frac{\sigma^2}{np}\left(\rho^2+\frac{\sigma^2}{np}\right) \\
&& + 4\mathbb{I}\{\|F_j\|\vee|G_j|\vee\|H_j\|>\rho\} \\
&\leq& \left(1+C_6\left(\rho+\sqrt{\frac{\sigma^2}{np}}+\eta\right)\right)|\im(a^{\H}bG_j)|^2 + C_7\eta^{-1}\left(\frac{\sigma^2}{n^2p}+\|F_j\|^2+\|H_j\|^2\right) \\
&& + C_5\frac{\sigma^2}{np}\left(\rho^2+\frac{\sigma^2}{np}\right) + 4\rho^{-2}\left(\|F_j\|^2+\|H_j\|^2\right) + 4\mathbb{I}\{|G_j|>\rho\},
\end{eqnarray*}
for some $\eta$ to be specified later, 
where the last inequality is by (\ref{eq:homepod}). Summing over $j\in[n]$, we obtain
\begin{eqnarray*}
n\ell(\wh{z},z^*) &\leq& \sum_{j=1}^n|\wh{z}_j-dz_j^*|^2 \\
&\leq& \left(1+C_6\left(\rho+\sqrt{\frac{\sigma^2}{np}}+\eta\right)\right)\sum_{j=1}^n|\im(a^{\H}bG_j)|^2 + C_7\eta^{-1}\frac{\sigma^2}{np} + C_5\frac{\sigma^2}{p}\left(\rho^2+\frac{\sigma^2}{np}\right) \\
&& + (C_7\eta^{-1}+4\rho^{-2})\sum_{j=1}^n\left(\|F_j\|^2+\|H_j\|^2\right) + 4\sum_{j=1}^n\mathbb{I}\{|G_j|>\rho\}.
\end{eqnarray*}
By the same argument that leads to the bound (\ref{eq:F-bound})-(\ref{eq:tenergy}) (with $\gamma=\frac{\sigma^2}{np}$ in (\ref{eq:tenergy})), we have
\begin{eqnarray*}
\sum_{j=1}^n(\|F_j\|^2+\|H_j\|^2) &\leq& C'\frac{\log n}{np}\fnorm{\wh{V}-bz^{*\H}}^2 \leq C'\frac{\log n}{np}\frac{\sigma^2}{p}, \\
\sum_{j=1}^n\mathbb{I}\{|G_j|>\rho\} &\leq& \frac{4\sigma^2}{\rho^2p} \exp\left(-\frac{1}{16}\sqrt{\frac{\rho^2 np}{\sigma^2}}\right), \\
\sum_{j=1}^n|\im(a^{\H}bG_j)|^2 &\leq& \left(1+C''\left(\eta+\eta^{-1}\frac{\sigma^2}{np}+\sqrt{\frac{\log n}{np}}\right)\right)\frac{\sigma^2}{2p}.
\end{eqnarray*}
Take $\eta=\rho^2=\sqrt{\frac{\log n+\sigma^2}{np}}$, and we have some constant $C'''>0$ such that
$$\ell(\wh{z},z^*)\leq \left(1+C'''\left(\frac{\log n+\sigma^2}{np}\right)^{1/4}\right)\frac{\sigma^2}{2np}.$$
Note that the above bound is derived from conditions (\ref{eq:hollow1})-(\ref{eq:hollow7}), and thus the result holds with high probability.
\end{proof}

\begin{proof}[Proof of Theorem \ref{thm:main-z2}]
The first conclusion is an immediate consequence of Lemma \ref{lem:critical-z2}, Lemma \ref{lem:stat-error-z2} and Lemma \ref{lem:SDP-crude-z2}. By Lemma \ref{lem:loss-relation}, we also obtain the second conclusion. For the last conclusion, we have $|\wh{z}_j-z_j^*|\leq 2|\sqrt{n}u_j-z_j^*|$ and $|\wh{z}_j+z_j^*|\leq 2|\sqrt{n}u_j+z_j^*|$ by the definition of $\wh{z}_j$. Then,
$$\ell(\wh{z},z^*)\leq 4\left(\|u-z^*/\sqrt{n}\|^2\wedge\|u+z^*/\sqrt{n}\|^2\right)\leq \frac{16}{n^2}\fnorm{\wh{Z}-z^*z^{*\T}}^2,$$
by Davis-Kahan theorem \citep{davis1970rotation}.
Thus, we can derive the third conclusion from the second one. Finally, when $\sigma^2<(1-\epsilon)\frac{np}{2\log n}$, we know from (\ref{eq:bounded-by-ind}) and Lemma \ref{lem:stat-error-z2} that $\ell(\wh{V},z^*)=0$. Lemma \ref{lem:loss-relation} implies that $\fnorm{\wh{Z}-z^*z^{*\T}}^2=0$ and thus $\wh{Z}=z^*z^{*\T}$ is a rank-one matrix.
\end{proof}

\subsection{Proof of Theorem \ref{thm:lower-z2}}\label{sec:pf-lower}

Since $\ell(\wh{z},z)=2\left(1-\frac{1}{n}|\wh{z}^{\T}z|\right)$ and $n^{-2}\fnorm{\wh{z}\wh{z}^{\T}-zz^{\T}}^2=2\left(1-\frac{1}{n^2}|\wh{z}^{\T}z|^2\right)$, we have
$$n^{-2}\fnorm{\wh{z}\wh{z}^{\T}-zz^{\T}}^2=\ell(\wh{z},z)\left(1+\frac{1}{n}|\wh{z}^{\T}z|\right)\leq 2\ell(\wh{z},z),$$
and thus
\begin{eqnarray*}
\inf_{\wh{z}\in\{-1,1\}^n}\sup_{z\in\{-1,1\}^n}\mathbb{E}_z\ell(\wh{z},z) &\geq& \inf_{\wh{z}\in\{-1,1\}^n}\sup_{z\in\{-1,1\}^n}\mathbb{E}_z\frac{1}{2n^2}\fnorm{\wh{z}\wh{z}^{\T}-zz^{\T}}^2 \\
&\geq& \frac{1}{2}\inf_{\wh{Z}\in\mathbb{R}^{n\times n}}\sup_{z\in\{-1,1\}^n}\mathbb{E}_z\frac{1}{n^2}\fnorm{\wh{Z}-zz^{\T}}^2.
\end{eqnarray*}
It suffices to prove a lower bound for the loss $\fnorm{\wh{Z}-zz^{\T}}^2$. We lower bound the minimax risk by a Bayes risk
\begin{eqnarray*}
&& \inf_{\wh{Z}\in\mathbb{R}^{n\times n}}\sup_{z\in\{-1,1\}^n}\mathbb{E}_z\frac{1}{n^2}\fnorm{\wh{Z}-zz^{\T}}^2 \\
&\geq& \inf_{\wh{Z}\in\mathbb{R}^{n\times n}}\frac{1}{2^n}\sum_{z\in\{-1,1\}^n}\mathbb{E}_z\frac{1}{n^2}\fnorm{\wh{Z}-zz^{\T}}^2 \\
&\geq& \frac{1}{n^2}\sum_{1\leq j\neq k\leq n}\frac{1}{2^{n-2}}\sum_{z_{-(j,k)}\in\{-1,1\}^{n-2}}\inf_{\wh{T}}\frac{1}{4}\sum_{z_j\in\{-1,1\}}\sum_{z_k\in\{-1,1\}}\mathbb{E}_z|\wh{T}-z_jz_k|^2,
\end{eqnarray*}
where $z_{-(j,k)}$ is a sub-vector of $z$ by excluding the $j$th and the $k$th entries. For each $z_{-(j,k)}$, we have
\begin{eqnarray*}
&& \inf_{\wh{T}}\sum_{z_j\in\{-1,1\}}\sum_{z_k\in\{-1,1\}}\mathbb{E}_z|\wh{T}-z_jz_k|^2 \\
&\geq& \inf_{\wh{T}}\left(\mathbb{E}_{(z_{-(j,k)},z_j=1,z_k=-1)}|\wh{T}+1|^2 + \mathbb{E}_{(z_{-(j,k)},z_j=1,z_k=1)}|\wh{T}-1|^2\right) \\
&\geq& 2\int \diff\mathbb{P}_{(z_{-(j,k)},z_j=1,z_k=-1)}\wedge\diff\mathbb{P}_{(z_{-(j,k)},z_j=1,z_k=1)},
\end{eqnarray*}
where the last inequality is due to the classical Le Cam's two-point method. The total variation affinity characterizes the optimal testing error between two simple hypotheses of $z_k=-1$ versus $z_k=1$ with the values of all other parameters are known. By Neyman-Pearson lemma, we have
\begin{eqnarray*}
&& \int \diff\mathbb{P}_{(z_{-(j,k)},z_j=1,z_k=-1)}\wedge\diff\mathbb{P}_{(z_{-(j,k)},z_j=1,z_k=1)} \\
&\geq& \mathbb{P}_{(z_{-(j,k)},z_j=1,z_k=-1)}\left(\frac{\diff\mathbb{P}_{(z_{-(j,k)},z_j=1,z_k=1)}}{\diff\mathbb{P}_{(z_{-(j,k)},z_j=1,z_k=-1)}}>1\right) \\
&=& \mathbb{P}_{(z_{-(j,k)},z_j=1,z_k=-1)}\left(\sum_{j\in[n]\backslash\{k\}}z_jA_{jk}Y_{jk}>0\right) \\
&=& \mathbb{P}\left(\sigma\sum_{j\in[n]\backslash\{k\}}z_jA_{jk}W_{jk}>\sum_{j\in[n]\backslash\{k\}}A_{jk}\right).
\end{eqnarray*}
Let $\mathcal{A}$ be the collections of $A$'s that satisfy the conclusions of Lemma \ref{lem:ER-graph}, and we know that $\mathbb{P}(\mathcal{A})\geq 1-n^{-10}$. Let $\mathbb{P}_A$ be the shorthand of the conditional probability $\mathbb{P}(\cdot|A)$. For each $A\in\mathcal{A}$, a standard Gaussian tail bound implies
$$\mathbb{P}_A\left(\sigma\sum_{j\in[n]\backslash\{k\}}z_jA_{jk}W_{jk}>\sum_{j\in[n]\backslash\{k\}}A_{jk}\right)\geq \exp\left(-(1+\delta)\frac{np}{2\sigma^2}\right),$$
where $\delta=C\sqrt{\frac{\log n + \sigma^2}{np}} $ for some constant $C>0$. This implies
\begin{eqnarray*}
&& \mathbb{P}\left(\sigma\sum_{j\in[n]\backslash\{k\}}z_jA_{jk}W_{jk}>\sum_{j\in[n]\backslash\{k\}}A_{jk}\right) \\
&\geq& \inf_{A\in\mathcal{A}}\mathbb{P}_A\left(\sigma\sum_{j\in[n]\backslash\{k\}}z_jA_{jk}W_{jk}>\sum_{j\in[n]\backslash\{k\}}A_{jk}\right)\mathbb{P}(\mathcal{A}) \\
&\geq& \frac{1}{2}\exp\left(-(1+\delta)\frac{np}{2\sigma^2}\right).
\end{eqnarray*}
Therefore,
\begin{eqnarray*}
\inf_{\wh{z}\in\{-1,1\}^n}\sup_{z\in\{-1,1\}^n}\mathbb{E}_z\ell(\wh{z},z) &\geq& \frac{1}{2}\inf_{\wh{Z}\in\mathbb{R}^{n\times n}}\sup_{z\in\{-1,1\}^n}\mathbb{E}_z\frac{1}{n^2}\fnorm{\wh{Z}-zz^{\T}}^2 \\
&\geq& \frac{1}{16}\exp\left(-(1+\delta)\frac{np}{2\sigma^2}\right).
\end{eqnarray*}
By absorbing the constant $1/16$ into the exponent, the proof is complete.

\subsection{Proofs of Lemma \ref{lem:SDP-crude}, Lemma \ref{lem:SDP-crude-z2}, and Lemma \ref{lem:stat-error-z2}}\label{sec:pf-ini}

\begin{proof}[Proof of Lemma \ref{lem:SDP-crude}]
By the definition of $\wh{Z}=\wh{V}^{\H}\wh{V}$, we have $\Tr((A\circ Y)\wh{Z})\geq \Tr((A\circ Y)z^*z^{*\H})$. Rearranging this inequality, we obtain
\begin{equation}
\Tr(z^*z^{*\H}(z^*z^{*\H}-\wh{Z}))\leq \Tr\left((A\circ Y/p -z^*z^{*\H})(\wh{Z}-z^*z^{*\H})\right).\label{eq:SDP-1step-complex}
\end{equation}
The right hand side of (\ref{eq:SDP-1step-complex}) can be bounded by
\begin{eqnarray*}
&& \left|\Tr\left((A\circ Y/p -z^*z^{*\H})\wh{Z}\right)\right| + \left|\Tr\left((A\circ Y/p -z^*z^{*\H})z^*z^{*\H}\right)\right| \\
&\leq& \opnorm{A\circ Y/p -z^*z^{*\H}}\Tr(\wh{Z}) + \opnorm{A\circ Y/p -z^*z^{*\H}}\Tr(z^*z^{*\H}) \\
&=& 2n\opnorm{A\circ Y/p -z^*z^{*\H}} \\
&\leq& 2n\left(\frac{1}{p}\opnorm{(A-\mathbb{E}A)\circ z^*z^{*\H}} + \frac{\sigma}{p}\opnorm{A\circ W}\right).
\end{eqnarray*}
By Lemma \ref{lem:ER-graph},
\begin{eqnarray*}
\opnorm{(A-\mathbb{E}A)\circ z^*z^{*\H}} &=& \sup_{\|u\|=1}\left|\sum_{1\leq j\neq k \leq n}(A_{jk}-p)z_j^*\bar{z}^*_ku_j\bar{u}_k\right| \\
&\leq& \opnorm{A-\mathbb{E}A} \\
&\leq& C_1\sqrt{np},
\end{eqnarray*}
with probability at least $1-n^{-10}$.
By Lemma \ref{lem:bandeira}, $\opnorm{A\circ W}\leq C_2\sqrt{np}$ with probability at least $1-n^{-10}$. Thus, we have
$$\Tr(z^*z^{*\H}(z^*z^{*\H}-\wh{Z}))\leq C_3n\sqrt{\frac{(1+\sigma^2)n}{p}}.$$
Define $m=\frac{1}{n}\sum_{j=1}^n\wh{V}_j{z}_j^*$. By the inequality $\|x/\|x\|-y/\|y\|\|\leq 2\|x-y\|/\|x\|$, we have
\begin{eqnarray*}
\ell(\wh{V},z^*) &=& \min_{a\in\mathbb{C}^n:\|a\|^2=1}\frac{1}{n}\sum_{j=1}^n\|\wh{V}_j{z}_j^*-a\|^2 \\
&=& \min_{a\in\mathbb{C}^n\backslash\{0\}}\frac{1}{n}\sum_{j=1}^n\|\wh{V}_j{z}_j^*-a/\|a\|\|^2 \\
&\leq& \min_{a\in\mathbb{C}^n\backslash\{0\}}\frac{4}{n}\sum_{j=1}^n\|\wh{V}_j{z}_j^*-a\|^2 \\
&=& \frac{4}{n}\sum_{j=1}^n\|\wh{V}_j{z}_j^*-m\|^2 \\
&=& \frac{2}{n^2}\sum_{j=1}^n\sum_{l=1}^n\left(\|\wh{V}_j{z}_j^*-m\|^2+\|\wh{V}_l{z}_l^*-m\|^2\right) \\
&=& \frac{2}{n^2}\sum_{j=1}^n\sum_{l=1}^n\|\wh{V}_j{z}_j^*-\wh{V}_l{z}_l^*\|^2 \\
&=& \frac{4}{n^2}\sum_{j=1}^n\sum_{l=1}^n(1-\bar{z}_j^*{z}_l^*\wh{V}_j^{\H}\wh{V}_l) \\
&=& \frac{4}{n^2}\Tr(z^*z^{*\H}(z^*z^{*\H}-\wh{Z})).
\end{eqnarray*}
Therefore, we have $\ell(\wh{V},z^*)\leq 4C_3\sqrt{\frac{(1+\sigma^2)}{np}}$, and the proof is complete.
\end{proof}

\begin{proof}[Proof of Lemma \ref{lem:SDP-crude-z2}]
Following the same argument in the proof of Lemma \ref{lem:SDP-crude}, we have
$$\Tr(z^*z^{*\T}(z^*z^{*\T}-\wh{Z}))\leq Cn\sqrt{\frac{(1+\sigma^2)n}{p}},$$
with probability at least $1-n^{-9}$ and $\ell(\wh{V},z^*)\leq \frac{4}{n^2}\Tr(z^*z^{*\T}(z^*z^{*\T}-\wh{Z}))$. Then, we obtain the bound $\ell(\wh{V},z^*)\leq 4C\sqrt{\frac{(1+\sigma^2)}{np}}$, and the proof is complete.
\end{proof}

\begin{proof}[Proof of Lemma \ref{lem:stat-error-z2}]
Let $\mathcal{A}$ be the collections of $A$'s that satisfy the conclusions of Lemma \ref{lem:ER-graph}, and we know that $\mathbb{P}(\mathcal{A})\geq 1-n^{-10}$. Let $\mathbb{P}_A$ be the shorthand of the conditional probability $\mathbb{P}(\cdot|A)$. For each $A\in\mathcal{A}$, a standard Gaussian tail bound implies
\begin{eqnarray*}
\frac{8}{n}\sum_{j=1}^n\mathbb{P}_A\left(|U_j|>1-\delta\right) &\leq& \frac{16}{n}\sum_{j=1}^n\exp\left(-\frac{(1-\delta)^2(n-1)^2p^2}{2\sigma^2\sum_{k\in[n]\backslash\{j\}}A_{jk}}\right) \\
&\leq& \exp\left(-(1-\bar{\delta})\frac{np}{2\sigma^2}\right),
\end{eqnarray*}
where $\bar{\delta}=C\left(\delta+\sqrt{\frac{\log n}{np}}\right)$ for some constant $C>0$. Therefore,
\begin{eqnarray*}
&& \mathbb{P}\left(\frac{8}{n}\sum_{j=1}^n\mathbb{I}\{|U_j|>1-\delta\}>\exp\left(-\left(1-\bar{\delta}-\sqrt{\frac{2\sigma^2}{np}}\right)\frac{np}{2\sigma^2}\right)\right) \\
&\leq& \sup_{A\in\mathcal{A}}\mathbb{P}_A\left(\frac{8}{n}\sum_{j=1}^n\mathbb{I}\{|U_j|>1-\delta\}>\exp\left(-\left(1-\bar{\delta}-\sqrt{\frac{2\sigma^2}{np}}\right)\frac{np}{2\sigma^2}\right)\right) + \mathbb{P}(\mathcal{A}^c) \\
&\leq& \exp\left(-\sqrt{\frac{np}{\sigma^2}}\right) + n^{-10},
\end{eqnarray*}
by Markov's inequality. This immediately implies the first conclusion. For the second conclusion, it is easy to see that when $\left(1-\bar{\delta}-\sqrt{\frac{2\sigma^2}{np}}\right)\frac{np}{2\sigma^2}>\log n$, we have $\frac{8}{n}\sum_{j=1}^n\mathbb{I}\{|U_j|>1-\delta\}\leq \frac{1}{n}$, and thus the value of $\frac{8}{n}\sum_{j=1}^n\mathbb{I}\{|U_j|>1-\delta\}$ has to be $0$.
\end{proof}

\bibliographystyle{dcu}
\bibliography{reference}

\end{document}